\documentclass[reqno]{amsart}

\usepackage{amsmath,amssymb,verbatim,mathdots}
\usepackage{amsthm}
\usepackage{enumerate}
\usepackage[all]{xy}
\usepackage{float}
\usepackage{url}
\usepackage{wasysym}

\newcommand{\seq}{\subseteq}
\newcommand{\noit}[1]{\noindent\textit{#1}}
\newcommand{\claim}{\hfill$\dashv_{\text{\scriptsize{claim}}}$}
\newcommand{\udelta}{\underline{\delta}}
\newcommand{\smd}{\raisebox{.75pt}{\textrm{\scriptsize{~\!$\triangle$\!~}}}}
\newcommand{\nv}{\text{-}}
\newcommand{\mand}{\makebox[.4in]{and}}
\newcommand{\inv}{^{\nv 1}}
\newcommand{\minf}{\underline{\mathfrak{m}}}
\newcommand{\msup}{\overline{\mathfrak{m}}}
\newcommand{\mfrak}{\mathfrak{m}}

\newcommand{\abar}{\bar{a}}
\newcommand{\bbar}{\bar{b}}
\newcommand{\cbar}{\bar{c}}
\newcommand{\mbar}{\bar{m}}
\newcommand{\nbar}{\bar{n}}
\newcommand{\rbar}{\bar{r}}
\newcommand{\ubar}{\bar{u}}
\newcommand{\vbar}{\bar{v}}
\newcommand{\xbar}{\bar{x}}
\newcommand{\ybar}{\bar{y}}
\newcommand{\zbar}{\bar{z}}

\def\C{\mathbb C}

\def\N{\mathbb N}
\def\Q{\mathbb Q}
\def\R{\mathbb R}
\def\Z{\mathbb Z}

\newcommand{\cA}{\mathcal{A}}
\newcommand{\cB}{\mathcal{B}}
\newcommand{\cL}{\mathcal{L}}
\newcommand{\cM}{\mathcal{M}}
\newcommand{\cN}{\mathcal{N}}
\newcommand{\cZ}{\mathcal{Z}}

\def\Th{\operatorname{Th}}
\def\indd{\operatorname{ind}}
\def\ext{\operatorname{ext}}
\def\acl{\operatorname{acl}}
\def\tp{\operatorname{tp}}
\def\Fac{\operatorname{Fac}}
\def\Fib{\operatorname{Fib}}
\def\ap{\operatorname{ap}}

\newtheorem{theorem}{Theorem}
\newtheorem{lemma}[theorem]{Lemma}
\newtheorem{corollary}[theorem]{Corollary}
\newtheorem{proposition}[theorem]{Proposition}

\newtheorem{fact}[theorem]{Fact}

\newtheorem{alphatheorem}{Theorem}

\newtheorem{alphacorollary}[alphatheorem]{Corollary}

\theoremstyle{definition}
\newtheorem{definition}[theorem]{Definition}
\newtheorem{convention}[theorem]{Convention}
\newtheorem{example}[theorem]{Example}
\newtheorem{remark}[theorem]{Remark}
\newtheorem{question}[theorem]{Question}
\newtheorem{problem}[theorem]{Problem}

\title{Stability and sparsity in sets of natural numbers}

\author{Gabriel Conant}

\address{Department of Mathematics\\
University of Notre Dame\\
Notre Dame, IN, 46656, USA}

\email{gconant@nd.edu}

\begin{document}

\begin{abstract}

Given a set $A\seq\N$, we consider the relationship between stability of the structure $(\Z,+,0,A)$ and sparsity of the set $A$. We first show that a strong enough sparsity assumption on $A$ yields stability of $(\Z,+,0,A)$. Specifically, if there is a function $f\colon A\to\R^+$ such that $\sup_{a\in A}|a-f(a)|<\infty$ and $\{\frac{s}{t}:s,t\in f(A),~t\leq s\}$ is closed and discrete, then $(\Z,+,0,A)$ is superstable (of $U$-rank $\omega$ if $A$ is infinite). Such sets include examples considered by Palac\'{i}n and Sklinos \cite{PaSk} and Poizat \cite{PoZ}, many classical linear recurrence sequences (e.g. the Fibonaccci numbers), and any set in which the limit of ratios of consecutive elements diverges. Finally, we consider sparsity conclusions on sets $A\seq\N$, which follow  from model theoretic assumptions on $(\Z,+,0,A)$. We use a result of Erd\H{o}s, Nathanson, and S\'{a}rk\"{o}zy \cite{ENS} to show that if $(\Z,+,0,A)$ does not define the ordering on $\Z$, then the lower asymptotic density of any finitary sumset of $A$ is zero. Finally, in a theorem communicated to us by Goldbring, we use a result of Jin \cite{JinBD} to show that if $(\Z,+,0,A)$ is stable, then the upper Banach density of any finitary sumset of $A$ is zero.

\end{abstract}

\maketitle

\section{Introduction}
\setcounter{theorem}{0}
\numberwithin{theorem}{section}

The group of integers $(\Z,+,0)$ is a standard example of a very well-behaved superstable group. The structure of definable sets in $(\Z,+,0)$ is completely understood  (see Fact \ref{fact:Z}$(d)$), and generalizes to the model theoretic study of modules and $1$-based groups, which accounts for some of the earliest work in stability theory and model theory (see \cite{PiGST}, \cite{Prestbook}). On other hand, until recently very little was known about stable \emph{expansions} of $(\Z,+,0)$. Indeed, it was unknown if $(\Z,+,0)$ had any proper stable expansions, until 2014 when Palac\'{i}n and Sklinos \cite{PaSk} and Poizat \cite{PoZ} independently gave the first examples. Specifically, these authors considered the expansion of $(\Z,+,0)$ by a unary predicate for the powers of a fixed integer $q\geq 2$, which was shown to be superstable of $U$-rank $\omega$. Palac\'{i}n and Sklinos proved the same conclusion for other examples including the expansion by a predicate for the factorial numbers (see Fact \ref{fact:exPaSk}).

In this paper, we initiate a general study of stable expansions of $(\Z,+,0)$ obtained by adding a unary predicate for some distinguished subset $A\seq\Z$ (this expansion is denoted $\cZ_A:=(\Z,+,0,A)$). Thus our work is in the domain of following general problem (attributed to Goodrick in \cite{PaSk}).

\begin{problem}
Characterize the subsets $A\seq\Z$ such that $\Th(\cZ_A)$ is stable.
\end{problem}

At present, there is little evidence that a singular characterization of such sets should exist. It seems more likely that the resolution of this problem will manifest as a classification program in which naturally defined families of sets $A\seq\Z$, for which $\Th(\cZ_A)$ is stable, are isolated and studied. The aim of this paper is to begin this classification program. Our results will establish that for subsets $A\seq \N$ (or, more broadly, subsets of $\Z$ with either an upper bound or a lower bound), stability of $\Th(\cZ_A)$ is intimately tied to combinatorial and number theoretic properties of $A$. In particular, we first define a general robust condition on the asymptotic behavior of subsets $A\seq\N$, which encompasses the examples mentioned above and is sufficient to deduce stability of $\Th(\cZ_A)$. We then show that for $A\seq\N$, stability of $\Th(\cZ_A)$ implies that $A$ must be quite sparse with respect to asymptotic density. 

We now state our first main result, which deals with the first of the two tasks described above. A set $A\seq\Z$ is defined to be \emph{geometrically sparse} if there is  $f\colon A\to\R^+$ such that $\sup_{a\in A}|a-f(a)|<\infty$ and the set $\{\frac{s}{t}:s,t\in f(A),~t\leq s\}\seq[1,\infty)$ is closed and discrete. To reconcile this definition with our previous emphasis on subsets of $\N$, note that any geometrically sparse $A\seq\Z$ must be bounded below and so $\cZ_A$ and $\cZ_{A\cap\N}$ are interdefinable. The following is our main result.

\begin{alphatheorem}\label{Athm:GS}
If $A\seq\Z$ is geometrically sparse and infinite, then $\Th(\cZ_A)$ is superstable of $U$-rank $\omega$.
\end{alphatheorem}

In Section \ref{sec:GS}, we catalog many natural examples of geometrically sparse sets. This includes all of the examples considered in \cite{PaSk} (e.g. powers and factorials), as well as any set $A\seq\N$ such that, if $(a_n)_{n=0}^\infty$ is a strictly increasing enumeration of $A$, then $\lim_{n\rightarrow\infty}\frac{a_{n+1}}{a_n}=\infty$. Geometrically sparse sets also include many famous examples of linear recurrence sequences such as the Fibonacci numbers and Pell numbers (more generally, any recurrence sequence whose characteristic polynomial is the minimal polynomial of a Pisot number or a Salem number, see Section \ref{sec:GS}). We also emphasize that the class of geometrically sparse sets is extremely robust in the following sense: if $A\seq\Z$ is geometrically sparse and $F\seq\Z$ is finite, then \emph{any subset} of $A+F$ is geometrically sparse.

To prove Theorem \ref{Athm:GS}, we use general results on stable expansions by unary predicates, due to Casanovas and Ziegler \cite{CaZi} (which were also used by Palac\'{i}n and Sklinos in \cite{PaSk}). In particular, we show that for geometrically sparse $A\seq\Z$, stability of $\Th(\cZ_A)$ is equivalent to stability of the induced structure on $A$ from $(\Z,+,0)$. We then show this induced structure is superstable of $U$-rank $1$. The $U$-rank calculation in Theorem \ref{Athm:GS} is obtained by combining a result from \cite{PaSk} with the following theorem (proved in Section \ref{sec:Ugen}). We use $U(\cM)$ to denote the $U$-rank of a structure $\cM$.

\begin{alphatheorem}\label{Athm:MU}
Let $\cM$ be an infinite first-order $\cL$-structure, with $U(\cM)=1$, and fix an infinite subset $A\seq M$. Let $\cM_A$ denote the expansion of $\cM$ in the language $\cL_A$ obtained by adding a unary predicate for $A$. Let $A^{\indd}$ denote the induced structure on $A$ from all $0$-definable sets in $\cM$. If all $\cL_A$-formulas are bounded modulo $\Th(\cM_A)$, then $U(\cM_A)\leq U(A^{\indd})\cdot\omega$.
\end{alphatheorem}

The proof of this result uses similar techniques as in \cite[Theorem 2]{PaSk}, which considers the case $\cM=\cZ$ and $U(A^{\indd})=1$. The condition that all $\cL_A$-formulas are bounded modulo $\Th(\cM_A)$ is defined in Section \ref{sec:Ugen}, and is a key ingredient in the work of Casanovas and Ziegler \cite{CaZi}. We will work with this condition via the stronger notion of a \emph{sufficiently small} subset $A$ of a structure $\cM$, which is a slight variation on a similar property from \cite{CaZi}.  In Section \ref{sec:ssparse}, we characterize sufficiently small subsets $A$ of $(\Z,+,0)$ using the occurrence of nontrivial subgroups of $\Z$ in finitary sumsets of $\pm A=\{x\in\Z:|x|\in A\}$, as well as the lower asymptotic density of such sumsets (via a result of Erd\H{o}s, Nathanson, and S\'{a}rk\"{o}zy \cite{ENS} on infinite arithmetic progressions in sumsets of sets with positive density). From Theorem \ref{Athm:MU} we thus obtain the following stability test for structures of the form $\cZ_A$ (see Corollary \ref{cor:Acor}).

\begin{alphacorollary}\label{Acor}
Suppose $A\seq\Z$ is infinite and, for all $n>0$, the set $\{a_1+\ldots+a_k:k\leq n,~a_i\in\pm A\}$ does not contain a nontrivial subgroup of $\Z$. If $\Th(A^{\indd})$ is stable then $\Th(\cZ_A)$ is stable, with $\omega\leq U(\cZ_A)\leq U(A^{\indd})\cdot \omega$.
\end{alphacorollary}

In Section \ref{sec:GSstab} we begin the main technical work necessary for the proof of Theorem \ref{Athm:GS}, using Corollary \ref{Acor} as a guide. First, we modify an unpublished argument of Poonen to show that any geometrically sparse set $A\seq\Z$ satisfies the sumset assumption in Corollary \ref{Acor}. The rest of Section \ref{sec:GSstab} is devoted to interpreting $A^{\indd}$ in a superstable structure of $U$-rank $1$ (specifically, $\N$ with the successor function and unary predicates for all subsets of $\N$, see Corollary \ref{cor:Hex}). In Section \ref{sec:index}, we refine the analysis of $A^{\indd}$ for well-known examples of geometrically sparse sets (see Theorem \ref{thm:index}). For example, we use the Skolem-Mahler-Lech Theorem to show that if $A\seq\Z$ is a recurrence sequence of the kind discussed after Theorem \ref{Athm:GS}, then $A^{\indd}$ is interpreted in $\N$ with the successor function and unary predicates for all arithmetic progressions. 

Theorem \ref{Athm:GS} establishes a sparsity assumption on sets $A\seq\N$ which is sufficient for stability of $\Th(\cZ_A)$. In Section \ref{sec:sstable}, we turn to the second task mentioned above, which concerns sparsity assumptions \emph{necessary} for stability. We prove:

\begin{alphatheorem}
Fix $A\seq\N$. 
\begin{enumerate}[$(a)$]
\item If $\cZ_A$ does not define the ordering on $\Z$ then any sumset $A+\ldots+A$ has lower asymptotic density $0$.
\item If $\Th(\cZ_A)$ is stable then any sumset $A+\ldots+A$ has upper Banach density $0$.
\end{enumerate}
\end{alphatheorem}

The two parts of this result are proved in Section \ref{sec:sstable} as Theorems \ref{thm:stabudelta} and \ref{thm:stabBD}, respectively. Part $(a)$ uses the previously mentioned result of Erd\H{o}s, Nathanson, and S\'{a}rk\"{o}zy \cite{ENS}, and part $(b)$ uses a result of Jin \cite{JinBD} on sumsets of sets with positive upper Banach density.  The proof of part $(b)$ was suggested to us by Goldbring, and  is included here with his permission.

\section{Expansions of first-order structures by unary predicates}\label{sec:Ugen}
\setcounter{theorem}{0}
\numberwithin{theorem}{section}

Throughout this section, we fix a first-order language $\cL$. Let $\cL^*$ be the language consisting of relations $R_\varphi(\xbar)$, where $\varphi(\xbar)$ is an $\cL$-formula over $\emptyset$ in variables $\xbar$. Let $\cM$ be an $\cL$-structure with universe $M$.

\begin{definition}\label{def:indd}
Fix a subset $A\seq M$.
\begin{enumerate}
\item Let $\cL_A=\cL\cup\{A\}$ where, abusing notation, we use $A$ for a unary relation symbol not in $\cL$.
\item Let $\cM_A$ denote the unique $\cL_A$-structure, with underlying universe $M$, satisfying the following properties:
\begin{enumerate}[$(i)$]
\item $\cM$ is the reduct of $\cM_A$ to $\cL$,
\item the unary relation $A$ is interpreted in $\cM_A$ as the subset $A$.
\end{enumerate}
\item Let $A^{\indd}$ be the $\cL^*$-structure, with universe $A$, such that $R_\varphi$ interpreted as $\varphi(A^n)$. 
\item An $\cL_A$-formula $\varphi(\xbar)$ is \textbf{bounded modulo $\Th(\cM_A)$} if it is equivalent, modulo $\Th(\cM_A)$, to a formula of the form
$$
Q_1z_1\in A\ldots Q_mz_m\in A \big[\psi(\xbar,z_1,\ldots,z_m)\big],
$$
where $Q_1,\ldots,Q_m$ are quantifiers and $\psi(\xbar,\zbar)$ is an $\cL$-formula. 
\end{enumerate}
\end{definition}

We will use the following result of Casanovas and Ziegler \cite{CaZi}, which relates the stability of $\Th(\cM_A)$ to stability of $\Th(\cM)$ and $\Th(A^{\indd})$.

\begin{fact}\textnormal{\cite[Proposition 3.1]{CaZi}}\label{fact:CaZi}
Let $\cM$ be an $\cL$-structure and fix $A\seq M$. Suppose every $\cL_A$-formula is bounded modulo $\Th(\cM_A)$. Then, for any cardinal $\lambda\geq|\Th(\cM)|$, $\Th(\cM_A)$ is $\lambda$-stable if and only if $\Th(\cM)$ is $\lambda$-stable and $\Th(A^{\indd})$ is $\lambda$-stable.
\end{fact}

Casanovas and Ziegler also provide a way to show that formulas are bounded.

\begin{definition}\label{def:locsmall}
Let $\cM$ be an $\cL$-structure. A subset $A\seq M$ is \textbf{sufficiently small} if, for any $\cL$-formula $\varphi(x,\ybar,\zbar)$, there is a model $\cN\models\Th(\cM_A)$ such that, for any $\bbar\in N^{\zbar}$, if $\{\varphi(x,\abar,\bbar):\abar\in A(N)^{\ybar}\}$ is consistent then it is realized in $\cN$.
\end{definition}

\begin{fact}\textnormal{\cite[Proposition 2.1]{CaZi}}\label{fact:CaZi2}
Let $\cM$ be an $\cL$-structure and fix a sufficiently small subset $A\seq M$. If $\Th(\cM)$ is stable and $\cM$ does not have the finite cover property over $A$, then every $\cL_A$-formula is bounded modulo $\Th(\cM_A)$. 
\end{fact}

\begin{remark}\label{rem:locsmall}
See \cite{CaZi} for the definition of ``$\cM$ has the finite cover property over $A$''.  The notion of ``sufficiently small'' is slightly weaker than what is used in \cite{CaZi}, where the authors work with so-called \emph{small} sets. The reader may verify that the proof of \cite[Proposition 2.1]{CaZi} only requires $A$ to be sufficiently small. Similar variations of these notions, and the relationships between them, are considered in \cite{BaBa}.
\end{remark}

In later results, we will calculate the $U$-rank of certain expansions of the group of integers by unary sets. Therefore, the goal of the rest of this section is to identify a general relationship between the $U$-ranks of the structures $\cM$, $A^{\indd}$, and $\cM_A$, under certain assumptions. For the rest of this section, we fix a subset $A\seq M$ such that every $\cL_A$-formula is bounded modulo $\Th(\cM_A)$. Let $\cM_A^*$ denote a $\kappa$-saturated monster model of $\Th(\cM_A)$, for some sufficiently large $\kappa$. Let $A^*=A(M^*)$. We then have the $\cL^*$-structure $(A^*)^{\indd}$ as given by Definition \ref{def:indd}. We use the following notation for the various type spaces arising in this situation.
\begin{enumerate}[$(i)$]
\item Given $n>0$ and $B\subset M^*$, 
\begin{align*}
S_n(B) & \text{ is the set of $n$-types in $\cL$ over $B$ consistent with $\Th(\cM)$,}\\
S^A_n(B) & \text{ is the set of $n$-types in $\cL_A$ over $B$ consistent with $\Th(\cM_A)$,}\\
S^A_{n,A}(B) & \text{ is the set of $p\in S^A_n(B)$ such that $p\models A(x_i)$ for all $1\leq i\leq n$.}
\end{align*}
\item Given $n>0$ and $B\subset A^*$,
\begin{align*}
S_n^{\indd}(B) & \text{ is the set of $n$-types in $\cL^*$ over $B$ consistent with $\Th(A^{\indd})$.}
\end{align*}
\end{enumerate}

By assumption on $A$, we may assume that types in $S^A_n(B)$ only contain $\cL_A$-formulas of the form $Q_1 z_1\in A\ldots Q_m z_m\in A\big(\psi(\xbar,\zbar)\big)$, for some $\cL$-formula $\psi(\xbar,\zbar)$ and quantifiers $Q_1,\ldots,Q_m$.

\begin{definition}
$~$
\begin{enumerate}
\item Fix an $\cL$-formula $\psi(\xbar,\zbar)$ and consider the $\cL_A$-formula
$$
\varphi(\xbar):=Q_1z_1\in A\ldots Q_m z_m\in A\big(\psi(\xbar,\zbar)\big),
$$
where each $Q_i$ is a quantifier. Define the $\cL^*$-formula
$$
R_\varphi(\xbar):=Q_1z_1\ldots Q_nz_n\big(R_\psi(\xbar,\zbar)\big).
$$
\item Fix $B\subset A^*$ and $n>0$. Given $p\in S_{n,A}^A(B)$, define the $\cL^*$-type
$$
p^{\indd}=\{R_\varphi(\xbar,\bbar):\varphi(\xbar,\ybar)\in\cL_A,~\bbar\in B,~\text{and }\varphi(\xbar,\bbar)\in p\}.
$$
Given $p\in S^{\indd}_A(B)$, define the $\cL_A$-type
$$
p^{\ext}=\{\varphi(\xbar,\bbar):\varphi(\xbar,\ybar)\in\cL_A,~\bbar\in B,~\text{and }R_\varphi(\xbar,\bbar)\in p\}.
$$
\end{enumerate}
\end{definition}

The following remarks are routine exercises, which we leave to the reader.

\begin{proposition}\label{prop:inddMT}
$~$
\begin{enumerate}[$(a)$]
\item $\Th(A^{\indd})$ has quantifier elimination in the language $\{R_\varphi:\varphi\in\cL_A\}$. 
\item Given $B\subset A^*$, the map $p\mapsto p^{\indd}$ is a bijection from $S^A_{n,A}(B)$ to $S^{\indd}_n(B)$, with inverse $p\mapsto p^{\ext}$.
\item $(A^*)^{\indd}$ is a $\kappa$-saturated elementary extension of $A^{\indd}$.
\item Fix $C\seq B\subset A^*$, $p\in S_{n,A}^A(C)$, and $q\in S_{n,A}^A(B)$ such that $q$ is an extension of $p$. Then $q^{\indd}$ is an extension of $p^{\indd}$ and, moreover, $q$ is a dividing extension of $p$ if and only if $q^{\indd}$ is a dividing extension of $p^{\indd}$. 
\end{enumerate}
\end{proposition}

\begin{lemma}\label{lem:inddUrk}
If $B\subset A^*$ and $p\in S_{n,A}^A(B)$, then $U(p)\geq U(p^{\indd})$. Moreover, if $\Th(\cM)$ is stable, then $U(p)=U(p^{\indd})$.
\end{lemma}
\begin{proof}
We show, by induction on ordinals $\alpha$, that $U(p^{\indd})\geq\alpha$ implies $U(p)\geq\alpha$ and, if $\Th(\cM)$ is stable, then the converse holds as well. The $\alpha=0$ and limit $\alpha$ cases are trivial, so we fix an ordinal $\alpha$ and consider $\alpha+1$.

First, suppose $U(p^{\indd})\geq\alpha+1$ and fix a forking extension $q_0\in S^{\indd}_n(C)$ of $p^{\indd}$, with $B\seq C\subset A^*$, such that $U(q_0)\geq\alpha$. Let $q=q_0^{\ext}$. By parts $(c)$ and $(d)$ of Proposition \ref{prop:inddMT}, it follows that $q$ is a forking extension of $p$. By induction, $U(q)=U(q_0)\geq\alpha$. Therefore $U(p)\geq\alpha+1$.

Now assume $\Th(\cM)$ is stable and $U(p)\geq\alpha+1$. By Fact \ref{fact:CaZi}, $\Th(\cM_A)$ is stable. We may fix a small model $N\models \Th(\cM_A)$ containing $B$, and $q\in S_{n,A}^A(N)$, such that $q$ is a forking extension of $p$ with $U(q)\geq\alpha$. Let $A'=A(N)$, and note that $B\seq A'\subset A^*$. Let $q'=q|_{A'}\in S^A_{n,A}(A')$ and let $r\in S^A_{n,A}(N)$ be a nonforking extension of $q'$ to $N$. Then $A(x)\in q$, $A(x)\in r$, and $q|_{A'}=q'=r|_{A'}$. By \cite[Theorem 12.30]{Pobook}, it follows that $q=r$, and so $U(q)=U(r)=U(q')$. Altogether, we may replace $q$ with $q'$ and assume $q\in S^A_{n,A}(A')$ is a forking extension of $p$, with $U(q)\geq \alpha$. By Proposition \ref{prop:inddMT}$(d)$, $q^{\indd}$ is a forking extension of $p^{\indd}$. Since $A'\subset A^*$, it follows by induction that $U(q^{\indd})\geq\alpha$. Therefore $U(p^{\indd})\geq\alpha+1$.
\end{proof}

\begin{definition}
Define $U(A)=\sup\{U(p):p\in S^A_{1,A}(\emptyset)\}$. 
\end{definition}

To clarify, $U(A)$ is the $U$-rank of the formula $A(x)$ with respect to $\Th(\cM_A)$ (i.e. the $U$-rank of the definable set $A^*$ in $\cM_A^*$). On the other hand, $U(A^{\indd})$ is the $U$-rank of $A^{\indd}$ as a structure (i.e. the $U$-rank of $\Th(A^{\indd})$).  

\begin{theorem}\label{thm:inddUrk}
Suppose  $\cM$ is a structure and $A\seq M$ is such that all $\cL_A$-formulas are bounded modulo $\Th(\cM_A)$. 
Then $U(A)\geq U(A^{\indd})$ and, if $\Th(\cM)$ is stable, then $U(A)=U(A^{\indd})$.
\end{theorem}
\begin{proof}
Recall that $U(A^{\indd})=\sup\{U(p):p\in S_1^{\indd}(\emptyset)\}$. Now apply Proposition \ref{prop:inddMT}$(c)$ and Lemma \ref{lem:inddUrk}.
\end{proof}

We now restate and prove Theorem \ref{Athm:MU}, which uses similar techniques as in the proof of \cite[Theorem 2]{PaSk} by Palac\'{i}n and Sklinos. In the following, we use $\cdot$ to denote standard ordinal multiplication, and $\otimes$ to denote ``natural multiplication" using Cantor normal form. 

\begin{theorem}\label{thm:UM}
Let $\cM$ be stable of $U$-rank $1$. If $A\seq M$ is infinite, and all $\cL_A$-formulas are bounded modulo $\Th(\cM_A)$, then $U(\cM_A)\leq U(A^{\indd})\cdot\omega$.
\end{theorem}
\begin{proof}
We may assume $\Th(A^{\indd})$ is superstable, since otherwise $U(A^{\indd})=\infty$ and the result holds trivially. By Fact \ref{fact:CaZi}, $\Th(\cM_A)$ is superstable. Let $\beta=U(A^{\indd})$. Let $T=\Th(\cM)$. Unless otherwise specified, we work in $\Th(\cM_A)$. 

\noit{Claim}: Given $c\in M^*$, if $c\in\acl(A^*B)$ then $U(c/B)<\beta\cdot\omega$.

\noit{Proof}: Let $\abar\in A^*$ be a finite tuple witnessing $c\in\acl(A^*B)$. Note $U(A^*)=\beta$ by Theorem \ref{thm:inddUrk} and so, using Lascar's inequality, $U(\abar/B)\leq U(\abar)\leq \beta\otimes |\abar|$. Since $A$ is infinite, we have $\beta>0$ and so $\beta\otimes|\abar|<\beta\cdot\omega$ . Applying Lascar's inequality once more, we obtain $U(c/B)\leq U(\abar/B)<\beta\cdot\omega$. \claim

We now show $U(\cM_A)\leq\beta\cdot\omega$. In particular, we fix $p\in S_1^A(M)$, and show $U(p)\leq\beta\cdot\omega$. To do this, we fix $B\supseteq M$ and a forking extension $q\in S_1(B)$ of $p$, and show $U(q)<\beta\cdot\omega$. After replacing $q$ by some nonforking extension, we may also assume $B$ is a model. Fix $b\in M^*$ realizing $q$. Let $a\in M^*$ realize a nonforking extension of $p$ to $B$. If $a\in\acl(A^*B)$ then $U(p)=U(a/B)<\beta\cdot\omega$ by the claim, and the result follows. So we may assume $a\not\in\acl(A^*B)$, which means $U^T(a/A^*B)=1$. For a contradiction, suppose $b\not\in\acl(A^*B)$. Then $b\not\in\acl(B)$, and so $U^T(b/B)=1=U^T(a/B)$. Since $\tp_{\cL}(a/M)=p=\tp_{\cL}(b/M)$, and $M$ is a model, it follows from stationarity in $T$ that $\tp_{\cL}(a/B)=\tp_{\cL}(b/B)$. Similarly, $b\not\in\acl(A^*B)$ implies $U^T(b/A^*B)=1$ and so, since $B$ is a model,  $\tp_{\cL}(a/A^*B)=\tp_{\cL}(b/A^*B)$ by stationarity in $T$. By induction on quantifiers over $A$, it follows that $a,b$ realize the same bounded $\cL_A$-formulas over $B$. By assumption on $A$, we have $\tp_{\cL_A}(a/B)=\tp_{\cL_A}(b/B)=q$, which contradicts that $q$ is a forking extension of $p$. Therefore $b\in\acl(A^*B)$, and so $U(q)=U(b/B)<\beta\cdot\omega$ by the claim.
\end{proof}

\section{Preliminaries on the group of integers}\label{sec:Z}

Let $\cZ=(\Z,+,0)$ denote the group of integers in the group language $\{+,0\}$. For $n\geq 2$, we let $\equiv_n$ denote equivalence modulo $n$, which is definable in $\cZ$ by the formula $\exists z(x=nz+y)$. For the rest of the paper, $\cL$ denotes the expanded language $\{+,0,(\equiv_n)_{n\geq 2}\}$. Note that an $\cL$-term is a linear combination of variables with positive integer coefficients. Since $=$ and $\equiv_n$ are invariant (modulo $\Th(\cZ)$) under addition, there is no harm in abusing notation and allowing negative integer coefficients in $\cL$-terms. We list some well known facts about $\Th(\cZ)$.

\begin{fact}\label{fact:Z}$~$
\begin{enumerate}[$(a)$]
\item $\Th(\cZ)$ has quantifier elimination in the language $\cL$ (see, e.g., \cite{Mabook}).

\item $\Th(\cZ)$ is superstable of $U$-rank $1$, but not $\omega$-stable (see, e.g., \cite{BPW}, \cite{Mabook}).

\item $\cZ$ does not have the finite cover property. In particular, for any $A\seq\Z$, $\cZ$ does not have the finite cover property over $A$. See \cite{BauFCP} and \cite{CaZi}.

\item Given $n>0$, the Boolean algebra of $\cZ$-definable subsets of $\Z^n$ is generated by cosets of subgroups of $\Z^n$ (see \cite{Prestbook}). It follows that if $A\seq\Z$ is infinite and bounded above or bounded below, then $A$ is not definable in $\cZ$.
\end{enumerate}
\end{fact}

From now on, given a set $A\seq\Z$, we use $A^{\indd}$ to denote the induced structure on $A$ from $\cZ$. The next corollary follows from results of \cite{PaSk} and Theorem \ref{thm:UM}.

\begin{corollary}\label{cor:UZ}
Suppose $A\seq\Z$. Assume that $A$ is not definable in $\cZ$ and every $\cL_A$-formula is bounded modulo $\Th(\cZ_A)$. Then $\omega\leq U(\cZ_A)\leq U(A^{\indd})\cdot\omega$. In particular, if $U(A^{\indd})$ is finite then $U(\cZ_A)=\omega$.
\end{corollary}
\begin{proof}
Since $\cZ$ has $U$-rank $1$, the upper bound comes from Theorem \ref{thm:UM}. The lower bound is from the fact, due to Palac\'{i}n and Sklinos \cite{PaSk}, that $\cZ$ has no proper stable expansions of finite $U$-rank. 
\end{proof}

Finally, we record some notation and facts concerning asymptotic density and sumsets of integers. Let $\N=\{0,1,2,\ldots\}$ denote the set of \emph{nonnegative} integers. Given $n>0$, we let $[n]$ denote $\{1,\ldots,n\}$.

\begin{definition}
Suppose $A\seq\N$. 
\begin{enumerate}
\item Set $\pm A=\{x\in \Z:|x|\in A\}$.
\item Given $n\geq 1$, define $A(n)=|A\cap[n]|$. 
\item The \textbf{lower asymptotic density} of $A$ is
$$
\udelta(A)=\liminf_{n\rightarrow\infty}\frac{A(n)}{n}.
$$
\end{enumerate}
\end{definition}

Given two subsets $A,B\seq\Z$, we write $A\sim B$ if $A\smd B$ is finite.  The following facts are straightforward exercises.

\begin{fact}\label{fact:sym}
$~$
\begin{enumerate}[$(a)$]
\item If $A,B\seq\N$ and $A\sim B$, then $\udelta(A)=\udelta(B)$. In particular, $\udelta(A)=0$ for any finite $A\seq\N$.
\item If $A,B\seq\N$ and $A\seq B$, then $\udelta(A)\leq\udelta(B)$.
\item If $m,r\in\N$ and $m>0$, then $\udelta(m\N+r)=\frac{1}{m}$.
\end{enumerate}
\end{fact}

\begin{definition}
$~$
\begin{enumerate}
\item A subset $X\seq\Z$ is \textbf{symmetric} if it is closed under the map $x\mapsto \nv x$.
\item Given $X\seq\Z$ and $n\in\Z$, define $nX=\{nx:x\in X\}$.
\item Given $X,Y\seq\Z$, define $X+ Y=\{x+y:x\in X,~y\in Y\}$.
\item Given $X\seq\Z$ and $n\geq 1$, define $n(X)=X+ X+\ldots+ X$ ($n$ times).
\item Given $X\seq\Z$ and $n\geq 1$, define $\Sigma_n(X)=\bigcup_{k=1}^n k(X)$.
\end{enumerate}
\end{definition}

We also cite the following result of Erd\H{o}s, Nathanson, and S\'{a}rk\"{o}zy \cite{ENS}.

\begin{fact}\textnormal{\cite{ENS}}\label{fact:ENS}
Suppose $A\seq\N$ is such that $\udelta(A)>0$. Then, for some $n\geq 1$, $n(A)$ contains an infinite arithmetic progression (and so $\Sigma_n(A)$ does as well).
\end{fact}

\begin{remark}\label{rem:NaNa}
Fact \ref{fact:ENS} can be viewed as a variant of \emph{Schnirel'mann's basis theorem} \cite{Schn}, which says that if $A\seq\N$ is such that $\udelta(A)>0$ and $1\in A$ then $\Sigma_n(A)=\N$ for some $n\geq 1$. Nash and Nathanson \cite{NaNa} use this result to give the following strengthening of Fact \ref{fact:ENS}: if $\udelta(A)>0$ then, for some $n\geq 1$, $\Sigma_n(A)$ is cofinite in $a\N$, where $a=\gcd(A)$.
\end{remark}

\section{Sparse sets and bounded formulas}\label{sec:ssparse}

The goal of this section is to characterize sufficiently small subsets of $\Z$ using a combinatorial notion of sparsity. 

\begin{definition}
A set $A\seq\Z$ is \textbf{sufficiently sparse} if $\udelta(\Sigma_n(\pm A)\cap\N)=0$ for all $n\geq 1$.
\end{definition}

\begin{proposition}\label{prop:sparse}
Given $A\seq\Z$, the following are equivalent.
\begin{enumerate}[$(i)$]
\item $A$ is sufficiently sparse.
\item For all $n\geq 1$, $\Sigma_n(\pm A)$ does not contain a nontrivial subgroup of $\Z$.
\item For all $n\geq 1$, $\Sigma_n(\pm A)$ does not contain a coset of a nontrivial subgroup of $\Z$.
\end{enumerate}
\end{proposition}
\begin{proof}
$(iii)\Rightarrow (ii)$ is trivial; and $(i)\Rightarrow(ii)$ follows from Fact \ref{fact:sym}. 

$(ii)\Rightarrow(iii)$. Suppose $\Sigma_n(\pm A)$ contains the coset $m\Z+r$, with $m\geq 1$. Then $r\in\Sigma_n(\pm A)$, and so $m\Z\seq\Sigma_{2n}(\pm A)$. 

$(ii)\Rightarrow(i)$: Suppose that $\udelta(B)>0$, where $B=\Sigma_n(\pm A)\cap\N$ for some $n\geq 1$. By Fact \ref{fact:ENS}, there is some $k\geq 1$ such that $\Sigma_k(B)$ contains an infinite arithmetic progression $m\N+r$. Then $m\Z\seq\Sigma_{2k}(\pm B)\seq \Sigma_{2kn}(\pm A)$.
\end{proof}

\begin{lemma}\label{lem:sparse}
Suppose $X,F\seq\Z$ are symmetric, with $F$ finite. If $\udelta(X\cap\N)=0$ then $\udelta((X+ F)\cap\N)=0$. 
\end{lemma}
\begin{proof}
Let $B=X\cap\N$, $C=(B+ F)\cap\N$, and $D=(X+ F)\cap\N$. We want to show $\udelta(D)=0$. Note that
$$
C\seq D\seq C\cup(\{x\in X:\min F\leq x<0\}+ F),
$$
and so $C\sim D$. Therefore, it suffices to show $\udelta(C)=0$. 

Let $t=|F|$ and $s=\max F$. Given $n\geq 1$, consider the function from $C\cap[n]$ to $B\cap[n+s]$ sending $x$ to the least $b\in B$ such that $x=b+f$ for some $f\in F$. This function is $(\leq t)$-to-one, and so $C(n)\leq B(n+s)t$. Suppose, toward a contradiction, that $\udelta(C)>0$. Then there is some $\epsilon>0$ and $n_*\geq 1$ such that, for all $n\geq n_*$, $\inf_{m\geq n}\frac{C(m)}{m}\geq \epsilon$. For any $n\geq n_*$, we have
$$
\inf_{m\geq n}\frac{B(m)+s}{m}\geq\inf_{m\geq n}\frac{B(m+s)}{m}\geq\inf_{m\geq n}\frac{C(m)}{mt}\geq\frac{\epsilon}{t}.
$$
Therefore, if $n\geq\max\{n_*,\frac{2st}{\epsilon}\}$, then
$$
\inf_{m\geq n}\frac{B(m)}{m}\geq\inf_{m\geq n}\frac{B(m)+s}{m}-\frac{s}{n}\geq \frac{\epsilon}{t}-\frac{s}{n}\geq\frac{\epsilon}{2t}>0.
$$
It follows that $\udelta(B)>0$, which contradicts our assumptions. 
\end{proof}

\begin{lemma}\label{lem:sschar}
A set $A\seq\Z$ is sufficiently small if and only if it is sufficiently sparse.
\end{lemma}
\begin{proof}
First, fix a sufficiently sparse set $A\seq\Z$. Let $\varphi(x,\ybar,\zbar)$ be an $\cL$-formula. Fix a tuple $\bbar\in\Z^{|\zbar|}$ and let $\Delta(x)=\{\varphi(x,\abar,\bbar):\abar\in A^{|\ybar|}\}$. Assume $\Delta(x)$ is consistent. We want to show $\Delta(x)$ is realized in $\cZ$. By quantifier elimination, we may write $\varphi(x,\ybar,\zbar)$ as $\bigvee_{\alpha=1}^p\psi_{\alpha}(x,\ybar,\zbar)$, where each $\psi_\alpha(x,\ybar,\zbar)$ is a conjunction of atomic and negated atomic $\cL$-formulas. For any $\cL$-term $f(x,\ybar,\zbar)$ and integer $n>0$, the negation of $f(x,\ybar,\zbar)\equiv_n 0$ is equivalent to $\bigvee_{r=1}^{n-1}(f(x,\ybar,\zbar)-r\equiv_n 0)$. Therefore, by enlarging the tuple $\zbar$ of variables, and the tuple $\bbar$ of parameters, we may assume no $\psi_\alpha$ contains the negation of $f(x,\ybar,\zbar)\equiv_n 0$. 

Let $c^*$ be a realization of $\Delta(x)$ in some elementary extension of $\cZ$. Let $\sigma\colon A^{\ybar}\to [p]$ be such that $c^*$ realizes $\psi_{\sigma(\abar)}(x,\abar,\bbar)$. We may assume $\sigma$ is surjective. If some $\psi_\alpha$ contains a conjunct of the form $f(x,\ybar,\zbar)=0$, for some term $f(x,\ybar,\zbar)$ with nonzero coefficient on $x$, then $c^*$ is already in $\Z$ and the desired result follows. Altogether, we may assume that, for each $\alpha\in[p]$, $\psi_\alpha$ is of the form
\[
\bigwedge_{i\in I_\alpha}f_i(x,\ybar,\zbar)\neq 0\wedge\bigwedge_{j\in J_\alpha}g_j(x,\ybar,\zbar)\equiv_{n_j}0
\]
for some finite sets $I_\alpha,J_\alpha$, $\cL$-terms $f_i,g_j$, and integers $n_j>0$. Define the types
\begin{align*}
\Delta_0(x) &:= \{f_i(x,\abar,\bbar)\neq 0:i\in I_{\sigma(\abar)},~\abar\in A^{|\ybar|}\}\mand\\
\Delta_1(x) &:= \{g_j(x,\abar,\bbar)\equiv_{n_j}0:j\in J_{\sigma(\abar)},~\abar\in A^{|\ybar|}\}.
\end{align*}
Then $c^*$ realizes $\Delta_0(x)\cup\Delta_1(x)$ and $\Delta_0(x)\cup\Delta_1(x)\models\Delta(x)$. Since there are only finitely many integers $n_j$ appearing in $\Delta_1(x)$, and $\Delta_1(x)$ is consistent, we may use the Chinese Remainder Theorem to find $n\geq 1$ and $0\leq r<n$ such that $\Delta_1(x)$ is equivalent to $x\equiv_n r$. Since there are only finitely many terms $f_i$ appearing in $\Delta_0(x)$, and $\bbar$ is fixed, we may find integers $m_1,\ldots,m_t,k>0$ and $c_1,\ldots,c_t\in\Z$ such that
\[
\{m_ix-c_i\not\in\Sigma_k(\pm A):1\leq i\leq t\}\cup\{x\equiv_n r\}\models \Delta_0(x)\cup\Delta_1(x).\tag{$\dagger$}
\]
Set $B=\{x\in\Z:m_ix-c_i\in\Sigma_k(\pm A)\text{ for some $1\leq i\leq t$}\}$. By $(\dagger)$, it suffices to show $B$ does not contain $n\Z+r$. Let $m=m_1\cdot\ldots\cdot m_t$ and set $F=\{\frac{m}{m_1}c_1,\ldots,\frac{m}{m_t}c_t\}$. Given $x\in B$, there is some $1\leq i\leq t$ such that
\[
mx\in\frac{m}{m_i}(\Sigma_r(\pm A)+c_i)\seq\Sigma_{mr}(\pm A)+\frac{m}{m_i}c_i.
\]
Therefore $mB \seq C:=\Sigma_{mr}(\pm A)+F$. By Lemma \ref{lem:sparse}, we have $\udelta(C\cap\N)=0$, and so $C$ does not contain $mn\Z+mr$. So $B$ does not contain $n\Z +r$, as desired.

Conversely, suppose $A\seq\Z$ is not sufficiently sparse. Without loss of generality, assume $0\in A$. By Proposition \ref{prop:sparse}, there is some $n\geq 1$ such that $\Sigma_n(\pm A)$ contains $m\Z$ for some $m\geq 1$. Then 
\[
\Th(\cZ_A)\models\forall x(x\equiv_m0\rightarrow \exists \ybar\in A^n(x=y_1+\ldots+y_n)).\tag{$\dagger\dagger$}
\]
Let $\varphi(x,\ybar):=x\equiv_m 0\wedge x\neq y_1+\ldots+y_n$. Fix $\cN\models\Th(\cZ_A)$ and let $\Delta(x)=\{\varphi(x,\abar):\abar\in A(N)^k\}$. Then $\Delta(x)$ is consistent, since $mN$ is infinite, but $\Delta(x)$ is not realized in $\cN$ by $(\dagger\dagger)$. 
\end{proof}

Recall that $\Th(\cZ)$ is $\lambda$-stable if and only if $\lambda\geq 2^{\aleph_0}$ and, moreover, $\cZ$ does not have the finite cover property. Thus we may combine Lemma \ref{lem:sschar} with Facts \ref{fact:CaZi} and \ref{fact:CaZi2} to make the following conclusion.

\begin{corollary}\label{cor:sscor}
Suppose $A\seq\Z$ is sufficiently sparse. Then every $\cL_A$-formula is bounded modulo $\Th(\cZ_A)$. Moreover, for any $\lambda\geq 2^{\aleph_0}$, $\Th(\cZ_A)$ is $\lambda$-stable if and only if $\Th(A^{\indd})$ is $\lambda$-stable.
\end{corollary}

Note, in particular, that if $A\seq\Z$ is sufficiently sparse and infinite, then $A$ is not definable in $\cZ$ by Fact \ref{fact:Z}$(d)$. Therefore, combining Corollary \ref{cor:sscor} with Corollary \ref{cor:UZ}, we obtain Corollary \ref{Acor} from the introduction.

\begin{corollary}\label{cor:Acor}
If $A\seq\Z$ is sufficiently sparse and $\Th(A^{\indd})$ is stable, then $\Th(\cZ_A)$ is stable and $\omega\leq U(\cZ_A)\leq U(A^{\indd})\cdot\omega$.
\end{corollary}

Note that there are plenty of sets $A\seq\Z$, which are not sufficiently sparse, but such that $\Th(\cZ_A)$ is stable. In particular, no infinite $\cZ$-definable subset $A\seq\Z$ can be sufficiently sparse (by Fact \ref{fact:Z}$(d)$). For this reason, we will mainly focus on subsets $A\seq\Z$ that are either bounded above or bounded below. Given such a set $A$, we say the sequence $(a_n)_{n=0}^\infty$ is a \emph{monotonic enumeration of $A$} if $A=\{a_n:n\in\N\}$ and $(a_n)_{n\in\N}$ is either strictly increasing or strictly decreasing. Any set of this form is interdefinable, modulo $\cZ$, with a subset of $\N$ and so, since our main goal is understanding stability of $\Th(\cZ_A)$, we will often focus on subsets of $\N$. As we have seen above, if $A\seq\Z$ is sufficiently sparse, then $\cL_A$-formulas are bounded modulo $\Th(\cZ_A)$, and so stability of $\Th(\cZ_A)$ reduces to stability of $\Th(A^{\indd})$. On the other hand, as we will see in the last section, for a subset $A$ of $\N$, stability of $\Th(\cZ_A)$ already implies that $A$ must be fairly sparse. This motivates the following question.

\begin{question}\label{ques:ssbound}
Is there a set $A\seq\N$ such that $\Th(\cZ_A)$ is stable, but is not the case that every $\cL_A$-formula is bounded modulo $\Th(\cZ_A)$?
\end{question}

Finally, it is worth pointing out that whether a set $A\seq\N$ is sufficiently sparse is not controlled only by the growth rate of the sequence, by also by the sumset structure. For example, if $A\seq\N$ is enumerated by $a_n=2^n+n$ then, as a sequence, $A$ grows like the powers of $2$ (which are sufficiently sparse and yield a stable expansion of $\cZ$ \cite{PaSk}). But $A$ is not sufficiently sparse, since any positive integer $n$ is of the form $2a_n-a_{n+1}+1$, and so $\Sigma_3(\pm A)=\Z$. In fact, $\Th(\cZ_A)$ is unstable (see Question \ref{ques:multbig}).

\section{Induced structure}\label{sec:induced}

In the last section, we gave a combinatorial condition on sets $A\seq\Z$, which ensures $\cL_A$-formulas are bounded in $\cZ_A$. The next step toward determining stability of $\Th(\cZ_A)$ is to analyze the induced structure $A^{\indd}$.  First, we use quantifier elimination for $\Th(\cZ)$ in the language $\cL$ to isolate a sublanguage $\cL^{\indd}$ of $\cL^*$ such that, for any $A\seq\Z$, $A^{\indd}$ is a definitional expansion of its reduct to $\cL^{\indd}$.

\begin{definition}
$~$
\begin{enumerate}
\item Let $\cL^{\indd}_0$ be the set of relations $R_\varphi\in \cL^*$, where $\varphi$ is either $x=0$ or an atomic $\cL$-formula of the form
\[
x_1+\ldots+x_k=y_1+\ldots+y_l,
\]
for some variables $x_1,\ldots,x_k,y_1,\ldots,y_l$.
\item Let $\cL^{\indd}=\cL^{\indd}_0\cup\{C_{m,r}:0\leq r<m<\omega\}$, where $C_{m,r}$ is the unary relation $R_{x\equiv_m r}\in\cL^*$.
\end{enumerate}
\end{definition}

\begin{proposition}\label{prop:unAP}
Suppose $A\seq\Z$. For any $\cL$-formula $\varphi(\xbar)$, there is a quantifier-free $\cL^{\indd}$-formula $\chi(\xbar)$ such that $A^{\indd}\models \forall\xbar(R_{\varphi}(\xbar)\leftrightarrow\chi(\xbar))$.
\end{proposition}
\begin{proof}
By quantifier elimination for $\Th(\cZ)$ with respect to the language $\cL$, we may assume $\varphi$ is an atomic $\cL$-formula. By construction of $\cL^{\indd}$, we may assume $\varphi(\xbar)$ is of the form 
\[
b_1x_1+\ldots+b_kx_k\equiv_n 0,
\]
 where $n>0$ and $b_1,\ldots,b_k\in\{0,\ldots,n-1\}$. Let $I$ be the set of tuples $\rbar\in\{0,\ldots,n-1\}^k$ such that there is $\abar\in\varphi(A^k)$ with $a_i\equiv_n r_i$ for all $i\in[k]$. Note that, for any $\rbar\in I$ and $\abar\in A^k$, if $a_i\equiv_n r_i$ for all $1\leq i\leq k$ then $\abar\in\varphi(A^k)$. Therefore
\[
A^{\indd}\models \forall \xbar\left(R_\varphi(\xbar)\leftrightarrow\bigvee_{\rbar\in I}\bigwedge_{i=1}^k C_{n,r_i}(x_i)\right).\qedhere
\]
\end{proof}

\begin{convention}
Given $A\seq\Z$, we use the previous proposition to identify the $\cL^*$-structure $A^{\indd}$ with its reduct to the language $\cL^{\indd}$.
\end{convention}

\begin{definition}
Given $A\seq\Z$, define $A^{\indd}_0$ to be the reduct of $A^{\indd}$ to $\cL^{\indd}_0$.
\end{definition}

To summarize the situation, consider a set $A\seq\Z$. We build $A^{\indd}$ in two steps. The first is $A^{\indd}_0$, which describes the $A$-solutions to homogeneous linear equations. Then $A^{\indd}$ is an expansion of $A^{\indd}_0$ by \emph{unary predicates} for $A\cap(m\Z+r)$, for all $0\leq r<m<\infty$. We will later show that for certain sets $A\seq\N$, $A^{\indd}_0$ is interpretable in a structure $\cN$ satisfying the property that \emph{any} expansion of $\cN$ by unary predicates is stable. This will allow us to conclude stability of $A^{\indd}$ without further analysis of the expansion of $A^{\indd}_0$ to $A^{\indd}$. The next definitions make this precise.

\begin{definition}
Given a first-order structure $\cM$ (in some language), we let $\cM^1$ denote the expansion of $\cM$ by unary predicates for all subsets of $M$.
\end{definition}

\begin{definition}
Let $\cA$ be an $\cL_1$-structure with universe $A$ and let $\cB$ be an $\cL_2$-structure with universe $B$. We say $\cA$ is an \textbf{interpretable reduct} of $\cB$ if there is a bijection $f\colon A\to B$ such that, for any $\cL_1$-formula $\varphi(x_1,\ldots,x_n)$, the set $f(\varphi(A^n))\seq B^n$ is $\cL_2(B)$-definable.
\end{definition}

Note that if a structure $\cA$ is an interpretable reduct of another structure $\cB$, then $\cA^1$ is an interpretable reduct of $\cB^1$. Thus we can summarize as follows.

\begin{corollary}\label{cor:monstH}
Suppose $A\seq\Z$. Then $A^{\indd}$ is an expansion of $A^{\indd}_0$ by unary predicates. Therefore, if $A^{\indd}_0$ is an interpretable reduct of a structure $\cN$, and $\cN^1$ is stable (of $U$-rank $\alpha$), then $A^{\indd}$ is stable (of $U$-rank at most $\alpha$).
\end{corollary}

In this paper, all stability results for structures of the form $A^{\indd}$, for some $A\seq\Z$, will be achieved via the previous corollary. So we ask the following question.

\begin{question}
Is there a set $A\seq\Z$ such that $A^{\indd}_0$ is stable and $A^{\indd}$ is unstable?
\end{question}

In \cite[Section 13.3]{Pobook}, it is observed that if the language of $\cM$ contains only equality then $\cM^1$ is stable. The next result gives more examples of such structures, which will be sufficient for our analysis.

\begin{proposition}\label{prop:unaryQE}
Suppose $\cM$ is an infinite first-order structure, in a language containing only a unary function symbol $s$, such that $s^\cM$ is an injective aperiodic function. Then $\cM^1$ is superstable of $U$-rank $1$.
\end{proposition}
\begin{proof}
Set $C=M\backslash s^\cM(M)$. Let $\hat{\cM}$ be the expansion of $\cM^1$ by  constant symbols for all elements of $C$ along with a unary function which is the inverse of $s^\cM$ on $s^\cM(M)$ and the identity on $C$. We will show that $\Th(\hat{\cM})$ has quantifier elimination. From this, it follows that $\Th(\hat{\cM})$ is \emph{quasi strongly minimal} and therefore superstable of $U$-rank $1$ (see \cite{BPW}). 

Let $\cN_1$ and $\cN_2$ be models of $\Th(\hat{\cM})$, and suppose $\cA$ is a common substructure. Fix a quantifier-free formula $\varphi(x,y_1,\ldots,y_n)$ and a tuple $\abar\in A^n$. Suppose $\cN_1\models\exists x\varphi(x,\abar)$. We want to show $\cN_2\models\exists x\varphi(x,\abar)$. We may assume $\varphi(x,\ybar)$ is of the form
$$
\psi(x)\wedge \bigwedge_{i=1}^n\bigwedge_{j=\nv k}^kx\neq s^j(y_i),
$$
where $k\in\N$ and $\psi(x)$ is the conjunction of $s\inv(x)\neq x$ and a finite Boolean combination of terms of the form $P(s^i(x))$, where $P$ is a unary relation and $i\in\Z$. Let $F=\{s^j(a_i):1\leq i\leq n,~\nv k\leq j\leq k\}\seq A$ (note $s^\cA=s^{\cN_1}|_A=s^{\cN_2}|_A$). We want to show $\psi(\cN_2)\backslash F\neq\emptyset$, and thus we may assume $\psi(\cN_2)$ is finite. Since $\psi(x)$ is a formula with no parameters, it follows that $|\psi(\cN_2)|=|\psi(\cN_1)|$. Since $\cA$ is a substructure of $\cN_1$ and $\cN_2$, and $\psi(x)$ is quantifier-free, we have $|F\cap\psi(\cN_1)|=|F\cap\psi(\cN_2)|$. Since $\cN_1\models\exists x\varphi(x,\abar)$, we have $|\psi(\cN_1)|>|F\cap\psi(\cN_1)|$. Altogether, this implies $|\psi(\cN_2)|>|F\cap\psi(\cN_2)|$, as desired.
\end{proof}

\section{Geometrically sparse sets}\label{sec:GS}

In this section, we define the notion of a \emph{geometrically sparse} subset $A\seq\Z$. Our ultimate goal is show that $\Th(\cZ_A)$ is stable whenever $A$ is geometrically sparse. Part of the motivation for the definition is to give a common generalization of the following examples analyzed in \cite{PaSk}. 

\begin{fact}\textnormal{(Palac\'{i}n and Sklinos \cite{PaSk})}\label{fact:exPaSk}
Suppose $A$ is the set $\Fac:=\{n!:n\in\N\}$ of factorial numbers or the set of powers
\[
\Pi(q_1,\ldots,q_t):=\left\{q_1^{\iddots^{ q_t^n}}:n\in\N\right\}
\]
where $q_1,\ldots,q_t\geq 2$ and $t\geq 1$.  Then $\Th(\cZ_A)$ is superstable of $U$-rank $\omega$.
\end{fact}

\begin{definition}\label{def:gs}
$~$
\begin{enumerate}
\item A set $S\seq\R^+$ is \textbf{geometric} if $\{\frac{s}{t}:s,t\in S,~t\leq s\}$ is closed and discrete.

\item A set $A\seq\Z$ is \textbf{geometrically sparse} if there is a function $f\colon A\to\R^+$ such that $f(A)$ is geometric and $\sup_{a\in A}|a-f(a)|<\infty$.
\end{enumerate}
\end{definition}

Note that, for any real number $b>1$, the set $\{b^n:n\in\N\}$ is geometric. Therefore, we immediately see that the set of powers $\Pi(q)$, where $q\geq 2$, is geometrically sparse. The rest of the examples in Fact \ref{fact:exPaSk} are geometrically sparse due to the following observation.

\begin{proposition}\label{prop:esgs}
Suppose $A\seq\N$ is infinite and monotonically enumerated $(a_n)_{n=0}^\infty$. If $\lim_{n\rightarrow\infty}\frac{a_{n+1}}{a_n}=\infty$ then $A$ is geometric, and thus geometrically sparse.
\end{proposition}
\begin{proof}
We may assume $0\not\in A$. Let $Q=\{\frac{a_n}{a_m}:m\leq n\}$. We want to show $Q$ is closed and discrete. In particular, we fix $b>1$ and show that the set $Q_0:=\{q\in Q:q\leq b\}$ is finite. By assumption on $A$, we may find some $N\geq 0$ such that if $n\geq N$ then $ba_{n-1}<a_n$. If $m<n$ and $\frac{a_n}{a_m}\leq b$, then $a_n\leq ba_m\leq ba_{n-1}$, and so $n<N$. Therefore $Q_0\seq\{\frac{a_n}{a_m}:m\leq n<N\}$. 
\end{proof}

The specific examples from Fact \ref{fact:exPaSk} can also be generalized as follows.

\begin{example}
Suppose $A\seq\N$ is monotonically enumerated $(a_n)_{n=0}^\infty$. If $\frac{a_{n+1}}{a_n}\in\N$ for all $n\in\N$, then $A$ is geometrically sparse.
\end{example}

The rest of this section is devoted to giving further examples of geometrically sparse sets. The first is somewhat ad hoc, but will be referenced again later.

\begin{example}\label{ex:stratio}
Fix real numbers $c>0$ and $b>1$. Given $n\in\N$, let $a_n$ be the integer part of $cb^n$. Then $A=\{a_n:n\in\N\}$ is geometrically sparse.

As a specific example, the set $\Fib$ of Fibonacci numbers is geometrically sparse (take $b=\frac{1+\sqrt{5}}{2}$ and $c=\frac{1}{\sqrt{5}}$). 
\end{example}

The next example shows that the class of geometrically sparse sets is invariant under ``finitary perturbations".

\begin{example}
Suppose $A\seq\Z$ is geometrically sparse and fix a finite set $F\seq\Z$. Then it is straightforward to show that any subset of $A+F$ is geometrically sparse. 
\end{example}

Using linear recurrence sequences, we can define a large family of geometrically sparse sets, which encompasses many classical and famous examples of integer sequences. We first define a special set of real algebraic numbers.

\begin{definition}
Let $U$ be the set of real algebraic numbers $\lambda>1$ such that, if $f(x)$ is the minimal polynomial of $\lambda$ over $\Q$, then $f(x)$ has integer coefficients and, if $\mu\in\C$ is a root of $f(x)$ distinct from $\lambda$, then $|\mu|\leq 1$.
\end{definition}

\begin{remark}
The set $U$ is commonly partitioned into the sets of \emph{Pisot numbers} and \emph{Salem numbers}, which are a well-studied topic in number theory. Specifically, given $\lambda\in U$, with minimal polynomial $f(x)$, $\lambda$ is a Pisot number if $|\mu|<1$ for all roots $\mu\neq\lambda$ of $f(x)$, and $\lambda$ is a Salem number if $|\mu|=1$ for at least one root $\mu\neq\lambda$ of $f(x)$. If $\lambda\in U$ then the sequence of fractional parts of the powers $\lambda^n$ exhibits interesting behavior in its distribution over the interval $[0,1]$. See \cite{PSbook} for details. 
\end{remark}

\begin{example}\label{ex:LHRR}
Fix $\lambda\in U$, and let $f(x)=x^{d+1}-c_dx^d-\ldots-c_1x-c_0$ be the minimal polynomial of $\lambda$ over $\Q$, where $d\geq 0$ and $c_0,\ldots,c_d\in\Z$.  Fix a tuple $\abar=(a_0,\ldots,a_d)\in\Z^{d+1}$ and recursively define, for $n\geq d$, the integer
\[
a_{n+1}=c_da_n+c_{d-1}a_{n-1}+\ldots+c_0a_{n-d}.
\]
Define the set $R(\abar;\cbar):=\{a_n:n\in\N\}$. We call $\lambda$ the \emph{Pisot-Salem number} of $R(\abar;\cbar)$ and we call $f(x)$ the \emph{characteristic polynomial} of $R(\abar;\cbar)$. Many well-known examples of recurrence relations can be realized in this way, including the \emph{Fibonacci numbers} $R(0,1;1,1)$, the \emph{Lucas numbers} $R(2,1;1,1)$, the \emph{Pell numbers} $R(0,1;1,2)$, the \emph{Pell-Lucas numbers} $R(2,2;1,2)$, the \emph{order $n$ Fibonacci numbers} $R(0,\mathinner{\overset{n-1}{\ldots}},0,1;1,\mathinner{\overset{n}{\ldots}},1)$ (see \cite{MilesGF}), the \emph{Padovan numbers} $R(1,1,1;1,1,0)$, and the \emph{Perrin numbers} $R(3,0,2;1,1,0)$. In the last two examples, the Pisot number is the largest real root of $x^3-x-1$, which is the smallest Pisot number (also called the \emph{plastic number}, with approximate value $1.32471$). The smallest \emph{known} Salem number (conjectured to be the smallest) is the largest real root of \emph{Lehmer's polynomial} $x^{10}+x^9-x^7-x^6-x^5-x^4-x^3+x+1$, with approximate value $1.17628$.

Now let $A=R(\abar;\cbar)$ be as above, with Pisot-Salem number $\lambda$ and characteristic polynomial $f(x)$. Assume $A$ is infinite. Let $\mu_1,\ldots,\mu_d\in\C$ be the pairwise distinct roots of $f(x)$ other than $\lambda$. By construction, there are $\alpha,\beta_1,\ldots,\beta_d\in\C$ such that, for all $n\in\N$, 
\[
a_n=\alpha\lambda^n+\beta_1\mu_1^n+\ldots+\beta_d\mu_d^n.
\]
Since $\lambda>1$ is real and $|\mu_i|\leq 1$ for all $1\leq i\leq d$, it follows that $\alpha$ is real (and nonzero since $A$ is infinite). Note  that $(a_n)_{n\in\N}$ is eventually strictly monotonic and so the geometric set $\{|\alpha|\lambda^n:n\in\N\}$ witnesses that either $A$ or $\nv A$ is geometrically sparse. Thus Theorem \ref{thm:mainGS}  will show that $\Th(\cZ_A)$ is superstable of $U$-rank $\omega$.
\end{example}

\begin{remark}
In the previous example, the expression of $a_n$ as a linear combination of integers powers of roots of the minimal polynomial of a number in $U$ is crucial to conclude that $A$ (or $\nv A$) is geometrically sparse. For example the set $A\seq\N$ enumerated by $a_n=3^n+2^n$ is not geometrically sparse. We would hypothesize that $\Th(\cZ_A)$ is stable for this and similar choices of $A$ (e.g., general linear homogeneous recurrence sequences with constant coefficients and distinct roots); but the methods here do not apply. 
\end{remark}

\section{Geometrically sparse sets are stable}\label{sec:GSstab}

The goal of this section is to prove Theorem \ref{Athm:GS}, which we now recall.

\begin{theorem}\label{thm:mainGS}
If $A\seq\Z$ is geometrically sparse and infinite then $\Th(\cZ_A)$ is superstable of $U$-rank $\omega$. 
\end{theorem}

We will prove this result in several steps, the first of which is to give a more precise description of infinite geometrically sparse sets. To avoid certain inconsequential annoyances, we will restrict to geometrically sparse subsets of $\Z^+$. This is sufficient to prove Theorem \ref{thm:mainGS} since if $A\seq\Z$ is geometrically sparse then $A\cap\Z^+$ is geometrically sparse, and $A$ is definable in $\cZ_{A\cap\Z^+}$ since $A$ is bounded below. Recall that given $A\seq\Z^+$ and $n\geq 1$, $A(n)$ denotes $|A\cap[n]|$.

\begin{proposition}\label{prop:GSexpl}
Suppose $A\seq\Z^+$ is geometrically sparse and monotonically enumerated $(a_n)_{n=0}^\infty$. Then there are real numbers $c>1$ and $\Theta\geq 0$, a real sequence $(\lambda_m)_{m=0}^\infty$, and  a weakly increasing surjective function $f\colon \N\to \N$ such that:
\begin{enumerate}[$(i)$]
\item $\{\lambda_m:m\in\N\}\seq\R^+$ is geometric,
\item $\lambda_{m+1}\geq c\lambda_m$ for all $m\in\N$, and
\item $|a_n-\lambda_{f(n)}|\leq\Theta$ for all $n\in\N$.
\end{enumerate}
In particular, $A(n)$ is $O(\log n)$. 
\end{proposition}
\begin{proof}
Let $g\colon A\to\R^+$ witness that $A$ is geometrically sparse, and set $r=\sup_{a\in A}|a-g(a)|$. Define a sequence $(b_m)_{m=0}^\infty$ such that $b_0=a_0$ and $b_{m+1}=\min \{a\in A:b_m+2r\leq a\}$.  For any $m\in\N$, we have $g(b_m)\leq b_m+r<b_{m+1}-r\leq g(b_{m+1})$. For $m\in\N$, set $\lambda_m=g(b_m)$. Let $S=\{\lambda_m:m\in\N\}$ and $C=\{b_m:m\in\N\}$. Then $C\seq A$ and $S\seq g(A)$. In particular, $S$ is geometric and $g\colon C\to S$ is a strictly increasing bijection. Since $S$ is geometric and $(\lambda_m)_{m=0}^\infty$ is strictly increasing, we may fix some $c>1$ such that $\frac{\lambda_{m+1}}{\lambda_m}\geq c$ for all $m\in\N$. 

Given $m\in\N$, let $A_m=\{a\in A:b_m\leq a<b_{m+1}\}$. Then $(A_m)_{m=0}^\infty$ is a partition of $A$ into nonempty $A$-convex sets with $\min A_m=b_m$, and $\max A_m-\min A_m\leq 2r$. Thus we obtain a well-defined weakly increasing surjective function $f\colon \N\to\N$ defined so that $a_n\in A_{f(n)}$. For $m,n\in\N$, we have $|a_n-b_{f(n)}|\leq 2r$ and $|b_m-\lambda_m|\leq r$ and so, setting $\Theta=3r$, properties $(i)$ through $(iii)$ are satisfied. It also follows that $\liminf_{m\to\infty}\frac{b_{m+1}}{b_m}\geq c$ and so, since $A=\bigcup_{m\in\N}A_m$, $\min A_m=b_m$, and $|A_m|$ is uniformly bounded, we have that $A(n)$ is $O(\log n)$. 
\end{proof}

The proof of Theorem \ref{thm:mainGS} proceeds according to the strategy discussed above (which is the same strategy employed by Palac\'{i}n and Sklinos \cite{PaSk} in Fact \ref{fact:exPaSk}). We first show that if $A\seq\N$ is geometrically sparse then $A$ is sufficiently sparse, and so stability of $\Th(\cZ_A)$ reduces to stability of $\Th(A^{\indd})$. We will then show that $A^{\indd}$ is superstable of $U$-rank $1$.

\subsection{Geometrically sparse sets are sufficiently sparse}\label{sec:poon}

The next two results are a mild deconstruction of an unpublished argument of Poonen \cite{PooMO}.\footnote{Specifically, the proof modifies (with Poonen's permission) the answer to a MathOverflow question (see \cite{PooMO}), which was asked in 2010 (not by the author).}

\begin{lemma}\label{lem:poon}
Suppose $X\seq \R^{\geq1}$ is closed and discrete, and let $Q=\{x\inv:x\in X\}$. Given $k\geq 1$, define
\[
Q_k=\left\{q_1+\ldots+q_k:q_i\in \pm Q,~|q_1|=1,~\sum_{i\in I}q_i\neq 0\text{ for all nonempty $I\seq [k]$}\right\}.
\]
Then, for all $k\geq 1$, there is some $\epsilon_k>0$ such that $|q|>\epsilon_k$ for all $q\in Q_k$.
\end{lemma}
\begin{proof}
 We proceed by induction on $k$. Note that $Q_1=\{1,\nv 1\}$, and so the base case is trivial. Assume the result for $Q_{k-1}$. Suppose, toward a contradiction, that there is a sequence $(x_n)_{n=0}^\infty$ in $Q_k$ converging to $0$. Let $x_n=q^n_1+\ldots+q^n_k$, where $q^n_i\in\pm Q$ and $1=|q^n_1|\geq\ldots\geq |q^n_k|$. 

We claim that $(q^n_k)_{n=0}^\infty$ converges to $0$. If not then, after replacing $(x_n)_{n=0}^\infty$ with a subsequence, we may fix $\epsilon>0$ such that $|q^n_k|>\epsilon$ for all $n\in\N$.  Let $Q^*=\{q\in Q:\epsilon<q\leq 1\}$. Then $Q^*$ is finite since $X$ is closed and discrete. Moreover, we have $|q^n_i|\in Q^*$ for all $i\in[k]$ and $n\in\N$. Therefore $\{x_n:n\in\N\}$ is finite. But each $x_n$ is nonzero since $0\not\in Q_k$, which contradicts $\lim_{n\rightarrow\infty} x_n=0$.

Now set $y_n=q^n_1+\ldots+q^n_{k-1}$. Then $(y_n)_{n<\omega}$ is a sequence of elements of $Q_{k-1}$ converging to $0$, which contradicts the induction hypothesis. 
\end{proof}

\begin{proposition}\label{prop:gsss}
If $A\seq\Z^+$ is geometrically sparse then $A$ is sufficiently sparse.
\end{proposition}
\begin{proof}
Let $f\colon A\to\R^+$ witness that $A$ is geometrically sparse.  For $k\geq 1$, let $Q_k$ and $\epsilon_k$ be as in Lemma \ref{lem:poon} with $X=\{\frac{s}{t}:s,t\in f(A),~t\leq s\}$.

\noit{Claim}: Given $n\geq 1$ there are $c,d>0$ such that any nonzero $a\in\Sigma_n(\pm A)$ can be written as $a=a_1+\ldots+a_k+q$, where $|q|\leq d$, $k\leq n$ and $a_i\in \pm A$ with $|a_i|< c|a|$. 

\noit{Proof}:  Let $d_0=\sup_{a\in A}|a-f(a)|$ and $c_0=\max\{\epsilon_k^{\nv 1}:1\leq k\leq n\}$. Set $d=nd_0$ and $c=c_0+c_0d+d_0$. Given $a\in A$, let $f(\nv a)=\nv f(a)$. Fix $a\in\Sigma_n(\pm A)$. We may write $a=f(a_1)+\ldots+f(a_n)+r_1+\ldots+r_n$ where $a_i\in \pm A$ and $|r_i|\leq d_0$. Set $t_i=f(a_i)$. Up to re-indexing, we may write $a=t_1+\ldots+t_k+r_1+\ldots+r_n$, where $k\leq n$, $|t_1|\geq\ldots\geq|t_k|$, and $\sum_{\in I}t_i\neq 0$ for all nonempty $I\seq[k]$. Let $q_*=r_1+\ldots+r_n$ and $q_i=\frac{t_i}{|t_1|}$. Then $q_1+\ldots+q_k\in Q_k$, and so  
\[
\frac{|a-q_*|}{|t_1|}=|q_1+\ldots+q_k|>\frac{1}{c_0}.
\]
Then, for any $i\in[k]$, we have $|t_i|\leq |t_1|<c_0|a-q_*|$. Therefore, for $i\in[k]$, we have 
\[
|a_i|=|t_i+r_i|<c_0|a|+c_0|q_*|+d_0\leq c_0|a|+c_0d+d_0\leq c|a|.
\]
Finally, if $q=r_{k+1}+\ldots+r_n$, then $a=a_1+\ldots+a_k+q$ and $|q|\leq d$. \claim

We now prove the result. Fix $n\geq 1$ and let $B=\Sigma_n(\pm A)\cap\N$. By Proposition \ref{prop:GSexpl}, $A(k)$ is $O(\log k)$, and so $B(k)$ is $O((\log k)^n)$ by the claim. So $\Sigma_n(\pm A)$ does not contain a nontrivial subgroup of $\Z$. By Proposition \ref{prop:sparse}, $A$ is sufficiently sparse.
\end{proof}

\subsection{Induced structure on geometrically sparse sets}

\begin{definition}
Let $\cN_{\mathfrak{s}}$ denote $\N$ with the successor function ${\mathfrak{s}}(x)=x+1$.
\end{definition}

Note that $\cN^1_{\mathfrak{s}}$ is superstable of $U$-rank $1$ by Proposition \ref{prop:unaryQE}. The goal of this section is show that if $A\seq\N$ is geometrically sparse, then $A^{\indd}$ is an interpretable reduct of $\cN^1_{\mathfrak{s}}$. To do this, it suffices by Corollary \ref{cor:monstH} to show $A^{\indd}_0$ is an interpretable reduct of $\cN^1_{\mathfrak{s}}$. This is the main work left for the proof of Theorem \ref{thm:mainGS}. The general strategy involves considering solutions in $A$ to linear equations of the form $x_1+\ldots+x_k=y_1+\ldots+y_l$, and separating these solutions into two categories based on whether the solution ``decomposes" into two solutions of linear equations in fewer variables, possibly at the cost of introducing a constant term to the equation. The next definition makes this precise.

\begin{definition}\label{def:Akl}
Suppose $A\seq\N$ is an infinite set with monotonic enumeration $(a_n)_{n=0}^\infty$. Fix integers $k,l\geq 0$ and $r\in\Z$.
\begin{enumerate}
\item Let $A(k,l,r)$ be the set of tuples $(\mbar,\nbar)\in\N^k\times\N^l$ such that 
\[
r+a_{m_1}+\ldots+a_{m_k}=a_{n_1}+\ldots+a_{n_l}.
\]
\item Given $s\in\N$, let $A(k,l,r,s)$ be the set of tuples $(\mbar,\nbar)\in \N^k\times\N^l$ such that, for some $s'\in\Z$, with $|s'|\leq s$, and some $I\subsetneq[k]$ and $J\subsetneq [l]$, with $I,J$ not both empty,
\[
r+s'+\sum_{i\in I}a_{m_i}=\sum_{j\in J}a_{n_j}\mand \sum_{i\not\in I}a_{m_i}=s'+\sum_{j\not\in J}a_{n_j}.
\]
Note that $A(k,l,r,s)\seq A(k,l,r)$ for any $s\in\Z$.
\end{enumerate}
\end{definition}

Given an infinite set $A\seq\N$, integers $k,l\geq 0$, and $r\in\Z$, if we  write $(\mbar,\nbar)\in A(k,l,r)$ then it is understood that $\mbar$ is a $k$-tuple and $\nbar$ is an $l$-tuple. Given a finite tuple $\ubar$ of integers we write $\max\ubar$ and $\min\ubar$ for the maximum and minimum coordinate of $\ubar$, respectively.

\begin{proposition}\label{prop:unarysucc}
Suppose $A\seq\N$ is infinite. For any $k,l,r,t\in\Z$, with $k,l\geq 0$, if $X_A$ is the set of $(\mbar,\nbar)\in A(k,l,r)$ such that $\max(\mbar,\nbar)-\min(\mbar,\nbar)\leq t$, then $X_A$ is definable in $\cN^1_{\mathfrak{s}}$.
\end{proposition}
\begin{proof}
Without loss of generality, we may assume $k\geq 1$. Let $\Sigma$ be the (finite) set of tuples $(u_1,\ldots,u_k,v_1,\ldots,v_l)$ of integers such that $|u_i|\leq t$ for all $i\leq k$ and $|v_j|\leq t$ for all $j\leq l$. Given $\sigma\in \Sigma$, let $X_\sigma$ be the set of $(\mbar,\nbar)\in X_A$ such that $m_i=m_1+u_i$ for all $i\leq k$ and $n_j=m_1+u_j$ for all $j\leq l$. By definition of $X_A$, we have $X_A=\bigcup_{\sigma\in\Sigma}X_\sigma$. Given $\sigma\in\Sigma$, let $P_\sigma$ be the set of $m\in\N$ such that $((m+u_1,\ldots,m+u_k),(m+v_1,\ldots,m+v_l))\in A(k,l,r)$. Then a tuple $(\mbar,\nbar)\in\N^k\times\N^l$ is in $X_\sigma$ if and only if it satisfies the formula
\[
\varphi(x_1,\ldots,x_k,y_1,\ldots,y_l):=x_1\in P_\sigma\wedge\bigwedge_{i=1}^kx_k=\mathfrak{s}^{u_i}(x_1)\wedge\bigwedge_{j=1}^l y_l={\mathfrak{s}}^{v_j}(x_1).
\]
So $X_\sigma$ is definable in $\cN^1_{\mathfrak{s}}$. Therefore $X_A$ is definable in $\cN^1_{\mathfrak{s}}$.
\end{proof}

The next lemma summarizes the main technical work required to analyze $A^{\indd}_0$ for geometrically sparse sets $A$.  For cleaner exposition, we only state the lemma for now, and postpone the proof until after we have used it to finish the proof of Theorem \ref{thm:mainGS}.

\begin{lemma}\label{lem:gsD}
Suppose $A\seq\Z^+$ is infinite and geometrically sparse. For any integers $k,l\geq 1$ and $r\in\Z$ there are $s,t\geq 0$ such that, for any $(\mbar,\nbar)\in A(k,l,r)$, if $\max(\mbar,\nbar)-\min(\mbar,\nbar)>t$ then $(\mbar,\nbar)\in A(k,l,r,s)$.
\end{lemma}

\begin{corollary}\label{cor:Hex}
If $A\seq\Z^+$ is infinite and geometrically sparse then $A^{\indd}$ is an interpretable reduct of $\cN^1_{\mathfrak{s}}$.
\end{corollary}
\begin{proof}
By Corollary \ref{cor:monstH}, it is enough to show that $A^{\indd}_0$ is an interpretable reduct of $\cN^1_{\mathfrak{s}}$.  In particular, we prove, by induction on integers $u\geq 0$, that for any $k,l\geq 0$ if $\max\{k,l\}=u$ then $A(k,l,r)$ is definable in $\cN^1_{\mathfrak{s}}$ for all $r\in \Z$. Note that, for any $k\geq 0$ and $r\in \Z$, both $A(k,0,r)$ and $A(0,k,r)$ are finite. So we may take our base case as $u=1$ and, in this case, assume $k=1=l$. For the base case, fix $r\in\Z$. If $(m,n)\in A(1,1,r)$ then $|m-n|\leq |r|$. Thus $A(1,1,r)$ is definable in $\cN^1_{\mathfrak{s}}$ by Proposition \ref{prop:unarysucc}. For the induction step, fix $u>0$ and assume that, for any $k,l\geq 0$, if $\max\{k,l\}<u$ then $A(k,l,r)$ is definable in $\cN^1_{\mathfrak{s}}$ for all $r\in \Z$. By induction on $0\leq v\leq u$, we show that, for any $k,l\geq 0$ if $\max\{k,l\}=u$ and $\min\{k,l\}=v$ then $A(k,l,r)$ is definable in $\cN^1_{\mathfrak{s}}$ for all $r\in \Z$. The base case $v=0$ follows as above. So fix $0<v\leq u$ and assume that, for all $k,l\geq 0$, if $\max\{k,l\}=u$ and $\min\{k,l\}<v$ then $A(k,l,r)$ is definable in $\cN^1_{\mathfrak{s}}$ for all $r\in \Z$. Fix $r\in\Z$.  Let $s,t\in\Z$ be as in Lemma \ref{lem:gsD}. Fix $k,l\geq 0$ such that $\max\{k,l\}=u$ and $\min\{k,l\}=v$. Let $\Sigma$ be the set of triples $(s',I,J)$ such that $s'\in\Z$, $I\subsetneq[k]$, and $J\subsetneq[l]$, with $|s'|\leq s$ and  $I$, $J$ not both empty. Given $(s',I,J)\in \Sigma$, let $A_{I,J}(k,l,r,s')$ be the set of $(\mbar,\nbar)\in\N^k\times\N^l$ such that
\[
((m_i)_{i\in I},(n_j)_{j\in J})\in A(|I|,|J|,r+s')\mand ((m_i)_{i\not\in I},(n_j)_{j\not\in J})\in A(k-|I|,l-|J|,\nv s').
\]
For any $(s',I,J)\in \Sigma$, we have $\max\{|I|,|J|\}<u$, $\max\{k-|I|,l-|J|\}\leq u$, and if $\max\{k-|I|,l-|J|\}=u$ then $\min\{k-|I|,l-|J|\}<v$. By both induction hypotheses, it follows that $A_{I,J}(k,l,r,s')$ is definable in $\cN^1_{\mathfrak{s}}$ for any $(s',I,J)\in\Sigma$ and $r\in \Z$. Let $X=X(k,l,r,t)$ be as in Proposition \ref{prop:unarysucc}. By Proposition \ref{prop:unarysucc} and Lemma \ref{lem:gsD}, we have 
\[
A(k,l,r)=X\cup \bigcup_{(s',I,J)\in\Sigma}A_{I,J}(k,l,r,s').
\]
Therefore $A(k,l,r)$ is definable in $\cN^1_{\mathfrak{s}}$ for any $r\in \Z$.  
\end{proof}

We now have all of the pieces necessary to prove the main result (modulo the proof of Lemma \ref{lem:gsD}, which is given in the next section). 

\begin{proof}[Proof of Theorem \ref{thm:mainGS}]
We may assume $A\seq\Z^+$. By Proposition \ref{prop:unaryQE} and Corollary \ref{cor:Hex}, $A^{\indd}$ is superstable of $U$-rank $1$.  By Proposition \ref{prop:gsss}, $A$ is sufficiently sparse. Altogether, $\Th(\cZ_A)$ is superstable of $U$-rank $\omega$ by Corollary \ref{cor:Acor}. 
\end{proof}

We have shown that if $A\seq\Z^+$ is geometrically sparse then $U(A^{\indd})=1$, which motivates the following question.

\begin{question}
For which ordinals $\alpha$ is there a set $A\seq\N$ such that $U(A^{\indd})=\alpha$? Is there $A\seq\N$ such that $\Th(A^{\indd})$ is strictly stable?
\end{question}

\subsection{Proof of Lemma \ref{lem:gsD}}\label{sec:agD}

Throughout this section, we fix an infinite geometrically sparse set $A\seq\Z^+$. Let $(a_n)_{n=0}^\infty$ monotonically enumerate $A$. We also fix a real sequence $(\lambda_m)_{m=0}^\infty$, real numbers $c>1$ and $\Theta\geq 0$, and a weakly increasing surjective function $f\colon\N\to\N$ as in Proposition \ref{prop:GSexpl}. For $n\in\N$, let $\theta_n=a_n-\lambda_{f(n)}$. We have $|\theta_n|\leq\Theta$ for all $n\in\N$. 

The proof of Lemma \ref{lem:gsD} requires a few preliminary steps.

\begin{proposition}\label{prop:agrowth}
There are real numbers $\delta>0$ and $b>1$ such that $\frac{a_n}{a_m}\geq \delta b^{n-m}$ for all $m,n\in\N$ with $m<n$.
\end{proposition}
\begin{proof}
Let $K=\lceil2\Theta\rceil+1$. For any $n\in\N$, we have $a_{n+K}>a_n+2\Theta$, and so $f(n+K)\geq f(n)+1$. Therefore $\lambda_{f(n+K)}\geq c\lambda_{f(n)}$ for all $n\in\N$. Given $n\in\N$ let $q(n)\geq 0$ and $r(n)\in\{0,\ldots,K-1\}$ be such that $n=q(n)K+r(n)$. For $m<n$, we have $n\geq m+q(n-m)K$, and so
\[
\lambda_{f(n)}\geq\lambda_{f(m+q(n-m)K)}\geq\lambda_{f(m)+q(n-m)}\geq c^{q(n-m)}\lambda_{f(m)}\geq \left(c^{\frac{1}{K}-1}c^{\frac{n-m}{K}}\right)\lambda_{f(m)}.
\]
So if we set $\delta_1=c^{\frac{1}{K}-1}$ and $b=c^{\frac{1}{K}}>1$ then $\frac{\lambda_{f(n)}}{\lambda_{f(m)}}\geq\delta_1b^{n-m}$ for any $m<n$.

For any $n\in\N$, we have $\frac{a_n}{\lambda_{f(n)}}=1+\frac{\theta_n}{\lambda_{f(n)}}$, and so $\lim_{n\rightarrow\infty}\frac{a_n}{\lambda_{f(n)}}=1$. Thus we may fix $\epsilon>0$ such that, for all $n\in\N$, $\epsilon<\frac{a_n}{\lambda_{f(n)}}<\frac{1}{\epsilon}$. Set $\delta=\delta_1\epsilon^2$. If $m<n$ then
\[
\frac{a_n}{a_m}\geq \frac{a_n}{\lambda_{f(n)}}\cdot\frac{\lambda_{f(n)}}{\lambda_{f(m)}}\cdot\frac{\lambda_{f(m)}}{a_m}\geq \delta b^{n-m}.\qedhere
\]
\end{proof}

\begin{proposition}\label{prop:sumfinish}
Fix integers $k,l\geq 1$ and $r\in\Z$. Set $s=\max\{|r|,\lceil(k+l)\Theta\rceil\}$. Fix $(\mbar,\nbar)\in A(k,l,r)$ and suppose there are $I_*\subseteq\{1,\ldots,k\}$ and $J_*\subseteq\{1,\ldots,l\}$ such that at least one of $I_*$ or $J_*$ is proper and $\sum_{i\in I_*}\lambda_{f(m_i)}=\sum_{j\in J_*}\lambda_{f(n_j)}$. Then $(\mbar,\nbar)\in A(k,l,r,s)$. 
\end{proposition}
\begin{proof}
Let $I=\{1,\ldots,k\}\backslash I_*$ and $J=\{1,\ldots,l\}\backslash J_*$. Define
\[
x=\sum_{i\in I}a_{m_i},~y=\sum_{j\in J}a_{n_j},\mand s'=\sum_{i\in I_*}\theta_{m_i}-\sum_{j\in J_*}\theta_{n_j}.
\]
Since $(\mbar,\nbar)\in A(k,l,r)$ and $\sum_{i\in I_*}\lambda_{f(m_i)}=\sum_{j\in J_*}\lambda_{f(n_j)}$, we have $x+r+s'=y$. In particular, $s'=y-x-r\in\Z$. Note also that $|s'|\leq (k+l)\Theta\leq s$. It follows that $(\mbar,\nbar)\in A(k,l,r,s)$ as desired. 
\end{proof}

\begin{lemma}\label{lem:AGl2}
Fix integers $k,l\geq 1$ and $r\in\Z$. Set $s=\max\{|r|,\lceil(k+l)\Theta\rceil\}$. For any $t\geq 0$, there is some $t_*=t_*(k,l,r,t)\geq 0$ such that, for any $(\mbar,\nbar)\in A(k,l,r)$, if $\max\mbar-\min\mbar\leq t$ then either $(\mbar,\nbar)\in A(k,l,r,s)$ or $\max(\mbar,\nbar)-\min(\mbar,\nbar)\leq t_*$. 
\end{lemma}
\begin{proof}
Let $\delta>0$ and $b>1$ be as in Proposition \ref{prop:agrowth}. Fix $r\in\Z$ and $t\geq 0$. Let $X$ be the set of $(\mbar,\nbar)\in A(k,l,r)$ such that  $m_k+t\geq m_1\geq\ldots\geq m_k$ and  $n_1\geq\ldots\geq n_l$. It suffices to find $t_*\geq 0$ such that, for all $(\mbar,\nbar)\in X$, either $(\mbar,\nbar)\in A(k,l,r,s)$ or $|m_1-n_j|\leq t_*$ for all $j\leq l$. Fix $(\mbar,\nbar)\in X$. We claim that $n_1\leq m_1+\log_b \frac{k}{\delta}$, which means $n_j\leq m_1+\log_b \frac{k}{\delta}$ for all $j\leq l$. If $n_1\leq m_1$ this is immediate, so we may assume $m_1<n_1$. We have
\[
a_{n_1}\leq a_{n_1}+\ldots+a_{n_l}=a_{m_1}+\ldots+a_{m_k}\leq ka_{m_1},
\]
and so $\frac{a_{n_1}}{a_{m_1}}\leq k$. By Proposition \ref{prop:agrowth}, $n_1-m_1\leq\log_b \frac{k}{\delta}$, as desired. 

Now, to prove the result, it suffices to construct $t_1,\ldots,t_l$ such that if $(\mbar,\nbar)\in X$ then either $(\mbar,\nbar)\in A(k,l,r,s)$ or $m_1\leq n_j+t_i$ for all $j\leq l$. Indeed, given this we may then define $t_*=\max\{t_1,\ldots,t_l,\log_b \frac{k}{\delta}\}$. We proceed by induction on $j$, treating $j=0$ as a vacuous base case.

Fix $w\in\{1,\ldots,l\}$ and suppose we have constructed $t_p$ for $p<w$. Let $\Sigma$ be the (finite) set of tuples $(i_1,\ldots,i_k,j_1,\ldots,j_{w-1})$ of integers such that $0\leq i_p\leq t$ for all $p\leq k$ and $|j_p|\leq t_p$ for all $p<w$. Given $\sigma\in\Sigma$, let $X_\sigma$ be the set of $(\mbar,\nbar)\in X$ such that $m_p=m_1-i_p$ for all $p\leq k$ and $n_p=m_1-j_p$ for all $p<w$. By induction $X=\bigcup_{\sigma\in\Sigma}X_\sigma$. We fix $\sigma\in\Sigma$ and construct $v_\sigma>0$ such that if $(\mbar,\nbar)\in X_\sigma$ then either $(\mbar,\nbar)\in A(k,l,r,s)$ or $m_1\leq n_w+v_\sigma$. Given this, we will then set $t_w=\max\{v_\sigma:\sigma\in\Sigma\}$. So fix $\sigma\in\Sigma$.

Given $m\geq\max\{i_1,\ldots,i_k,j_1,\ldots,j_{w-1}\}$, define
\[
d_m=r+(a_{m-i_1}+\ldots+a_{m-i_k})-(a_{m-j_1}+\ldots+a_{m-j_{w-1}}).
\]

\noit{Claim}: There is some $\epsilon>0$ such that, for any $(\mbar,\nbar)\in X_\sigma$, either $(\mbar,\nbar)\in A(k,l,r,s)$ or $\frac{d_{m_1}}{a_{m_1}}\geq\epsilon$.

Before proving the claim, we use it to finish the construction of $v_\sigma$. Let $\epsilon>0$ be as in the claim. Fix $(\mbar,\nbar)\in X_\sigma$. We have
\[
d_{m_1}=a_{n_w}+\ldots+a_{n_l}\leq (l-w+1)a_{n_w},
\]
and so, by the claim, either $(\mbar,\nbar)\in A(k,l,r,s)$ or 
\[
\frac{a_{m_1}}{a_{n_w}}\epsilon\leq\frac{a_{m_1}}{a_{n_w}}\left(\frac{d_{m_1}}{a_{m_1}}\right)=\frac{d_{m_1}}{a_{n_w}}\leq l-w+1.
\]
By Proposition \ref{prop:agrowth}, we may set $v_\sigma=\log_b\left(\frac{l-w+1}{\delta\epsilon}\right)$.

\noit{Proof of Claim}: Note that $\frac{d_{m_1}}{a_{m_1}}>0$ for any $(\mbar,\nbar)\in X_\sigma$, and so it suffices to find $\epsilon$ satisfying the result for sufficiently large $m_1$. Let $S=\{\lambda_m:m\in\N\}$, and recall that $S$ is geometric. Let $Q=\{\pm\frac{s_1}{s_2}:s_1,s_2\in S,~s_1\leq s_2\}$. Let $u=k+w-1$. By Lemma \ref{lem:poon}, we may fix $\epsilon>0$ such that, for any $q_1,\ldots,q_u\in Q$, if some $q_i=1$ and $\sum_{i\in Y}q_i\neq 0$ for all nonempty $Y\seq\{1,\ldots,u\}$, then $|q_1+\ldots+q_u|\geq 4\epsilon$.

Given $m\geq i_*:=\max\{i_1,\ldots,i_k,j_1,\ldots,j_{w-1}\}$, set 
\begin{align*}
\phi_m &=(\theta_{m-i_1}+\ldots+\theta_{m-i_k})-(\theta_{m-j_1}+\ldots+\theta_{m-j_{w-1}}),\mand\\
\eta_m &= (\lambda_{f(m-i_1)}+\ldots+\lambda_{f(m-i_k)})-(\lambda_{f(m-j_1)}+\ldots+\lambda_{f(m-j_{w-1})}).
\end{align*}
For $m\geq i_*$, we have $|\phi_m|\leq s$ and $d_m=r+\eta_m+\phi_m$. For $m\geq i_*$, let 
\[
p(m)=\max\{f(m-i_1),\ldots,f(m-i_k),f(m-j_1),\ldots,f(m-j_{w-1})\},
\]
and set $q_m=\frac{\eta_m}{\lambda_{p(m)}}$. We first show that, for any $(\mbar,\nbar)\in X_\sigma$, either $(\mbar,\nbar)\in A(k,l,r,s)$ or $q_{m_1}\geq 4\epsilon$. So fix $(\mbar,\nbar)\in X_\sigma$. We claim that if $q_{m_1}<4\epsilon$ then $(\mbar,\nbar)\in A(k,l,r,s)$. Indeed, if $0\leq q_{m_1}<4\epsilon$ then, by choice of $\epsilon$, there are nonempty $I_*\seq\{1,\ldots k\}$ and $J_*\seq\{1,\ldots,w-1\}$ such that $\sum_{i\in I_*}\lambda_{f(m_i)}=\sum_{j\in J_*}\lambda_{f(n_j)}$. So $(\mbar,\nbar)\in A(k,l,r,s)$ by  Proposition \ref{prop:sumfinish}. On the other hand, if $q_{m_1}<0$ then $\eta_{m_1}<0$, and so $a_{n_w}+\ldots+a_{n_l}=d_{m_1}\leq r+\theta_{m_1}\leq r+s$. Therefore $(\mbar,\nbar)\in A(k,l,r,s)$ witnessed by $I=\emptyset$, $J=\{w,\ldots,l\}$, and $s'=a_{n_w}+\ldots+a_{n_l}-r$. 

Finally, note that if $(\mbar,\nbar)\in X_\sigma$ then $i_1=0$, and so $f(m_1)\leq p(m_1)$. We may choose an integer $m_*>0$ such that if $m\geq m_*$ then
\[
\left|\frac{r+\phi_m}{\lambda_{f(m)}}\right|\leq\epsilon\mand a_m=\lambda_{f(m)}+\theta_m\leq 2\lambda_{f(m)}.
\]
Therefore, if $(\mbar,\nbar)\in X_\sigma\backslash A(k,l,r,s)$ and $m_1\geq m_*$ then, since $q_{m_1}\geq 4\epsilon$ by the above, we have
\[
\frac{d_{m_1}}{a_{m_1}}=\frac{r+\eta_{m_1}+\phi_{m_1}}{a_{m_1}}\geq\frac{\eta_{m_1}}{2\lambda_{f(m_1)}}-\left|\frac{r+\phi_m}{\lambda_{f(m)}}\right|\geq\frac{q_{m_1}}{2}-\epsilon\geq\epsilon.\qedhere
\]
\end{proof}

\setcounter{equation}{0}

We can now prove Lemma \ref{lem:gsD}.

\begin{proof}[Proof of Lemma \ref{lem:gsD}]
Fix $k,l\geq 1$ and $r\in\Z$. We want to find $s,t\geq 0$ such that, for any $(\mbar,\nbar)\in A(k,l,r)$, if $\max(\mbar,\nbar)-\min(\mbar,\nbar)>t$ then $(\mbar,\nbar)\in A(k,l,r,s)$.

Let $S=\{\lambda_m:m\in\N\}$, and recall that $S$ is geometric. Let $Q=\{\pm\frac{s_1}{s_2}:s_1,s_2\in S,~s_1\leq s_2\}$. By Lemma \ref{lem:poon}, we may fix $0<\epsilon<1$ such that, for any $w\leq k+l$ and any $q_1,\ldots,q_w\in Q$, if some $q_i=1$ and $\sum_{i\in X}q_i\neq 0$ for all nonempty $X\seq\{1,\ldots,w\}$, then $|q_1+\ldots+q_w|\geq\epsilon$. 

Let $k_*=\max\{k,l\}$. Choose $m_*>0$ such that, for all $m\geq m_*$,
\[
\frac{2k_*}{c^m}<\frac{\epsilon}{2},\hspace{10pt}\frac{2k_*\Theta+|r|}{\lambda_{f(m)}}<\frac{\epsilon}{2},\mand a_m\leq 2\lambda_{f(m)}.
\]
As in the proof of Proposition \ref{prop:agrowth}, we have $f(n+\lceil 2\Theta\rceil+1)\geq f(n)+1$ for all $n\in\N$. So we may choose $p>0$ such that for any $m,n$, if $m-n>p$ then $f(m)-f(n)>m_*$. Let $t_1=t'(k,l,r,kp)$ and $t_2=t'(l,k,r,lp)$ be as in Lemma \ref{lem:AGl2}. Set $t=\max\{t_1,t_2\}$ and $s=\max\{|r|,\lceil(k+l)\Theta\rceil\}$. 

Fix $(\mbar,\nbar)\in A(k,l,r)$ such that $\max(\mbar,\nbar)-\min(\mbar,\mbar)>t$. We want to show $(\mbar,\nbar)\in A(k,l,r,s)$. Without loss of generality, we may assume $m_1\leq\ldots\leq m_k$ and $n_1\leq\ldots\leq n_l$.  By choice of $t$ and Lemma \ref{lem:AGl2}, we may also assume $m_k> m_1+kp$ and $n_l> n_1+lp$. Therefore, we may fix $u\in\{1,\ldots,k-1\}$ and $v\in\{1,\ldots,l-1\}$ such that $m_u+p<m_{u+1}$ and $n_v+p<n_{v+1}$.

Without loss of generality, we may assume  that $a_{m_1}+\ldots+a_{m_u}\geq a_{n_1}+\ldots+a_{n_v}$. 
Define
\[
d=(a_{m_1}+\ldots+a_{m_u})-(a_{n_1}+\ldots+a_{n_v})+(\theta_{m_u+1}+\ldots+\theta_{m_k})-(\theta_{n_v+1}+\ldots+\theta_{n_l}).
\]
Note that if $u<i\leq k$ and $v<j\leq l$ then $m_i>m_u+p\geq m_*$ and $n_j>n_v+p\geq m_*$. Therefore 
\begin{equation}
|d|\leq ka_{m_u}+2k_*\Theta\leq 2k\lambda_{f(m_u)}+2k_*\Theta.
\end{equation}
Moreover, since $(\mbar,\nbar)\in A(k,l)$, we have
\begin{equation}
\lambda_{f(m_{u+1})}+\ldots+\lambda_{f(m_k)}+r=\lambda_{f(n_{v+1})}+\ldots+\lambda_{f(n_l)}-d.
\end{equation}
Let $m=\max\{m_k,n_l\}$. After dividing both sides of $(2)$ by $\lambda_{f(m)}$, and rearranging, we obtain
\begin{equation}
q_1+\ldots+q_w=\frac{d-r}{\lambda_{f(m)}},
\end{equation}
where $w=k+l-u-v$ and $q_i\in Q$, with at least one $q_i=1$. Since $m-m_u\geq m_k-m_u>p$, we have $f(m)-f(m_u)>m_*$ and so it follows from (1) that
\begin{equation}
\left|\frac{d-r}{\lambda_{f(m)}}\right|\leq \frac{2k \lambda_{f(m_u)}}{\lambda_{f(m)}}+\frac{2k_*\Theta+|r|}{\lambda_{f(m)}}<\frac{2k}{c^{m_*}}+\frac{\epsilon}{2}<\epsilon
\end{equation}
By $(3)$, $(4)$, and choice of $\epsilon$, there are nonempty $I_*\seq\{u+1,\ldots,k\}$ and $J_*\seq\{v+1,\ldots,l\}$ such that $\sum_{i\in I_*}\lambda_{f(m_i)}=\sum_{j\in J_*}\lambda_{f(n_j)}$. By Proposition \ref{prop:sumfinish},  $(\mbar,\nbar)\in A(k,l,r,s)$ as desired. 
\end{proof}

\subsection{Induced structure for special cases}\label{sec:index}

We have shown that if $A\seq\Z^+$ is infinite and geometrically sparse then $A^{\indd}$ is an interpretable reduct of $\cN^1_{\mathfrak{s}}$. In this section, we discuss how this can be refined for the examples given in Section \ref{sec:GS}. 

\begin{definition}
Let $\dot{\N}$ be the structure with universe $\N$ in the language of equality. Let $\cN^{\ap}$ and $\cN^{\ap}_{\mathfrak{s}}$ be the expansions of $\dot{\N}$ and $\cN_{\mathfrak{s}}$, respectively, by unary predicates for arithmetic progressions $m\N+r$, where $0\leq r<m<\omega$. 
\end{definition}

In \cite{PaSk}, Palac\'{i}n and Sklinos give a detailed description of $A^{\indd}$ for the examples in Fact \ref{fact:exPaSk}. In our terminology, this analysis can be summarized as follows.

\begin{fact}\textnormal{\cite{PaSk}}\label{fact:inddPaSk}
\begin{enumerate}[$(a)$]
\item If $A=\Fac$ then $A^{\indd}$ is interdefinable with $\dot{\N}$.
\item If $A=\Pi(q_1,\ldots,q_t)$, with $q_i\geq 2$ and $t\geq 1$, then $A^{\indd}_0$ is an interpretable reduct of $\cN_{\mathfrak{s}}$ and $A^{\indd}$ is an interpretable reduct of $\cN^{\ap}_{\mathfrak{s}}$. 
\end{enumerate}
\end{fact}

In this section, we sketch how to amend the proofs in Section \ref{sec:GSstab} to obtain the following result which, together with Remark \ref{rem:index}, generalizes Fact \ref{fact:inddPaSk}.

\begin{theorem}$~$\label{thm:index}
\begin{enumerate}[$(a)$]
\item Suppose $A\seq\N$ is monotonically enumerated $(a_n)_{n=0}^\infty$. If $\lim_{n\rightarrow\infty}\frac{a_{n+1}}{a_n}=\infty$ then $A^{\indd}_0$ is interdefinable with $\dot{\N}$.
\item Suppose $A=R(\abar;\cbar)$ for some $\abar,\cbar\in\Z^{d+1}$ as in Example \ref{ex:LHRR}. Then $A_0^{\indd}$ is an interpretable reduct of $\cN_{\mathfrak{s}}$ and $A^{\indd}$ is an interpretable reduct of $\cN^{\ap}_{\mathfrak{s}}$. 
\end{enumerate}
\end{theorem}
\begin{proof}
Part $(a)$. Without loss of generality, assume $0\not\in A$. 

\noit{Claim}: For any integers $k,l\geq 1$, with $\max\{k,l\}>1$, there is $t\geq 0$ such that, for any $(\mbar,\nbar)\in A(k,l,0)$ if $\max(\mbar,\nbar)>t$ then $(\mbar,\nbar)\in A(k,l,0,0)$. 

\noit{Proof}: Let $k_*=\max\{k,l\}$ and fix $t\geq 0$ such that if $n>t$ then $a_n>k_*a_{n-1}$. Fix $(\mbar,\nbar)\in A(k,l,0)$ such that $\max(\mbar,\nbar)>t$ and, without loss of generality, $m_1=\max\mbar\geq\max\nbar=n_1$. Then $(\mbar,\nbar)\in A(k,l,0)$ implies $a_{m_1}\leq k_*a_{n_1}$, and so $m_1=n_1$ since $m_1>t$. Let $I=[k]\backslash\{1\}$ and $J=[l]\backslash\{1\}$. Since $k_*>1$, $I$ and $J$ witness $(\mbar,\nbar)\in A(k,l,0,0)$ (see Definition \ref{def:Akl}).\claim

Now follow the proof of Corollary \ref{cor:Hex} to show that $A(k,l,0)$ is definable in $\dot{\N}$ for any $k,l\geq 0$. In the base case of the proof, replace the use of Proposition \ref{prop:unarysucc} by the observation that $A(1,1,0)=\{(n,n):n\in\N\}$ is definable in $\dot{\N}$. In the induction step, replace Lemma \ref{lem:gsD} by the above claim, and replace the use of  Proposition \ref{prop:unarysucc} by  the fact that, for any $k,l\geq 0$ and $t\geq 0$, if $X$ is the set of $(\mbar,\nbar)\in A(k,l,0)$ such that $\max(\mbar,\nbar)\leq t$, then $X$ is finite and thus definable in $\dot{\N}$.

Part $(b)$. First, after possibly replacing $A$ with $\nv A$ and removing a finite set, we may assume $A\seq\Z^+$ and $(a_n)_{n=0}^\infty$ is a monotonic enumeration of $A$. By the proof of Corollary \ref{cor:Hex}, in order to show $A^{\indd}_0$ is an interpretable reduct of $\cN_{\mathfrak{s}}$, it suffices to show that the conclusion of Proposition \ref{prop:unarysucc} holds with $\cN^1_{\mathfrak{s}}$ replaced by $\cN_{\mathfrak{s}}$. To do this, it suffices to show that the set $P_\sigma$, from the proof of Proposition \ref{prop:unarysucc}, is definable in $\dot{\N}$. Fix $\beta_0,\ldots,\beta_d,\mu_0,\ldots,\mu_d\in\C$ such that $a_n=\sum_{t=0}^d\beta_t\mu_t^n$ for all $n\in\N$. Let $\sigma=(\ubar,\vbar)$, and set $\gamma_t=\beta_t(\sum_{i=1}^k\mu_t^{u_i}-\sum_{j=1}^l\mu_t^{v_j})$. If $f\colon\N\to\C$ is such that $f(n)=\sum_{t=0}^d\gamma_t\mu_t^n$, then $P_\sigma$ is the zero-set of $f$. Moreover, $f(n)$ is a recurrence relation with the same characteristic polynomial as $A$. Therefore $f(n)$ either is identically zero or has finitely many zeroes by  the Skolem-Mahler-Lech Theorem \cite[Theorem 2.1]{EPSWbook} applied to \emph{weakly non-degenerate} recurrence relations (see \cite[Corollary 7.2]{SchLRS}). Thus $P_\sigma$ is definable in $\dot{\N}$, as desired. Finally, to conclude $A^{\indd}$ is an interpretable reduct of $\cN^{\ap}_{\mathfrak{s}}$, apply Proposition \ref{prop:unAP}, together with the fact that $A$ is eventually periodic modulo any fixed $n\geq 1$ (since it is given by a linear recurrence sequence).
\end{proof}

\begin{remark}\label{rem:index}$~$
 \begin{enumerate}[$(a)$]
  \item Lemma 4.6 of \cite{PaSk}, which proves the conclusion of Theorem \ref{thm:index}$(a)$ for the factorials, can be easily modified to give a different proof of Theorem \ref{thm:index}$(a)$.
  
 \item Suppose $A\seq\N$ is monotonically enumerated $(a_n)_{n=0}^\infty$, with $\lim_{n\rightarrow\infty}\frac{a_{n+1}}{a_n}=\infty$. Then, by Corollary \ref{cor:monstH}, $A^{\indd}$ is an interpretable reduct of an expansion of $\dot{\N}$ by unary predicates.

 \item If $A$ is the set of factorials then, since $A$ is eventually $0$ modulo any fixed integer $n\geq 1$, it follows that $A^{\indd}$ is a definitional expansion of $A^{\indd}_0$, which explains Fact \ref{fact:inddPaSk}$(a)$. Thus, we have the same result for any $A\seq\N$ such that $\lim_{n\rightarrow\infty}\frac{a_{n+1}}{a_n}=\infty$ and $(a_n)_{n=0}^\infty$ is eventually constant modulo any $n\geq 1$.
 
 \item If $A=\Pi(q_1,\ldots,q_t)$, with $t\geq 2$, then $A^{\indd}$ is an interpretable reduct of $\cN^{\ap}$ since $\lim_{n\rightarrow\infty}\frac{a_{n+1}}{a_n}=\infty$ and $A$ periodic modulo any $m\geq 1$. 
\end{enumerate}
\end{remark}

\section{Sparse consequences of stability}\label{sec:sstable}

The previous results have shown that $\Th(\cZ_A)$ is stable if $A$ is sparse enough. In this section, we study sparsity assumptions on $A$ which are necessary for stability of $\Th(\cZ_A)$. The first such result can be obtained by just considering the case that $\cZ_A$ does not define the ordering on $\Z$ (i.e. does not define the set $\N$).

\begin{definition}
A set $A\seq\N$ is \textbf{$\udelta$-sparse} if $\udelta(\Sigma_n(A))=0$ for all $n\geq 1$.
\end{definition}

\begin{remark}\label{rem:ENS}
By Fact \ref{fact:ENS}, a set $A\seq\N$ is $\udelta$-sparse if and only if, for all $n\geq 1$, $\Sigma_n(A)$ does not contain an infinite arithmetic progression.
\end{remark}

\begin{theorem}\label{thm:stabudelta}
Given $A\seq\N$, if $\cZ_A$ does not define $\N$ then $A$ is $\udelta$-sparse.
\end{theorem}
\begin{proof}
Since $\Sigma_n(A)$ is definable in $\cZ_A$ for all $n\geq 1$, it suffices (via Remark \ref{rem:ENS}) to show that if $A$ contains an infinite arithmetic progression then it is unstable. So assume $A$ contains $m\N+r$ for some $0\leq r<m$. Setting $X=\{z\in \Z:\lceil\frac{\nv r}{m}\rceil\leq z<0\}$, we have $z\in \N$ if and only if $mz+r\in \Sigma_n(A)$ and $z\not\in X$ (this uses the fact that $\Sigma_n(A)$ is bounded below by $0$). Therefore $\cZ_A$ defines $\N$.
\end{proof}

For ease of exposition, we make the following definition.

\begin{definition}
A set $A\seq\N$ is \textbf{stable} if $\Th(\cZ_A)$ is stable. 
\end{definition}

If $A\seq\N$ is stable then $\cZ_A$ does not define $\N$, and so $A$ is $\udelta$-sparse by Theorem \ref{thm:stabudelta}. With Remark \ref{rem:ENS} in mind, it is natural to ask about the behavior of finite arithmetic progressions in stable subsets of $\N$. 

\begin{definition}
A set $A\seq\N$ is \textbf{$\ap$-sparse} if, for all $n\geq 1$, $\Sigma_n(A)$ does not contain arbitrarily long finite arithmetic progressions.
\end{definition}

\begin{question}\label{ques:stabAP}
Suppose $A\seq\N$ is stable. Is $A$ $\ap$-sparse?
\end{question}

Szemer\'{e}di's Theorem \cite{Szem} implies that if $A\seq\N$ is $\ap$-sparse then, for all $n\geq 1$, $\Sigma_n(A)$ has upper Banach density zero. Therefore, a strategy toward a negative answer to Question \ref{ques:stabAP} would be to find an unstable $A\seq\N$ with positive upper Banach density. Theorem \ref{thm:stabBD} below, which is due to Goldbring, shows this is not possible (the statement and proof are included with his permission).

\begin{definition}\label{def:BD}
Fix $A\seq\N$.
\begin{enumerate}
\item The \textbf{upper Banach density of $A$} is 
\[
\delta^b(A)=\limsup_{n-m\to\infty}\frac{|A\cap [m+1,n]|}{n-m}.
\]
\item $A$ is \textbf{$\delta^b$-sparse} if $\delta^b(\Sigma_n(A))=0$ for all $n\geq 1$.
\end{enumerate}
\end{definition}

\begin{theorem}\label{thm:stabBD}
Given $A\seq\N$, if $A$ is stable then $A$ is $\delta^b$-sparse.
\end{theorem}
\begin{proof}[Proof \textnormal{(Goldbring)}] As in Theorem \ref{thm:stabudelta}, it suffices to show that if $A\seq\N$ is stable then $\delta^b(A)=0$. So fix $A\seq\N$ and assume $\delta^b(A)>0$. Without loss of generality, we may assume $1\in A$. Thus, by a result of Jin \cite[Theorem 1]{JinBD}, there is some $n\geq 1$ such that if $B=\Sigma_n(A)$ then $\delta^b(B)=1$ (i.e. $B$ contains arbitrarily large intervals). In particular, $\Z$ cannot be covered by finitely many translates of $\Z\backslash B$. Since $B\seq\N$, $\Z$ also cannot be covered by finitely many translates of $B$. Since $B$ is definable in $\cZ_A$, it follows that $\Th(\cZ_A)$ is unstable (in particular, $x-y\in B$ has the order property; see \cite[Lemma 5.1]{PoStG}). 
\end{proof}

It is also worth noting that, since $\udelta(A)\leq\delta^b(A)$ for any $A\seq\N$, Theorem \ref{thm:stabBD} strengthens (quite significantly) the conclusion of Theorem \ref{thm:stabudelta} applied to stable sets. Once again, a positive answer to Question \ref{ques:stabAP} would strengthen this even further, since $\ap$-sparse subsets of $\N$ are $\delta^b$-sparse by Szemer\'{e}di's Theorem. While we have been unable to answer Question \ref{ques:stabAP}, we can obtain some partial results on the behavior of finite arithmetic progressions in stable subsets of $\N$.

\begin{definition}
Given $A\seq\Z$, an arithmetic progression $(a+id)_{i=0}^{n-1}\seq A$ is \textbf{strongly contained in $A$} if $a-id\not\in A$ for all $1\leq i\leq n-1$.
\end{definition}

\begin{proposition}\label{prop:strcon}
Fix $A\seq\Z$ and suppose the formula $y-x\in A$ is stable in $\Th(\cZ_A)$. Then $A$ does not strongly contain arbitrarily long arithmetic progressions.
\end{proposition}
\begin{proof}
Suppose $A$ strongly contains arbitrarily long arithmetic progressions. Fix an integer $k>0$. We find $(b_i,c_i)_{1\leq i\leq k}$ in $\Z$ such that $c_j-b_i\in A$ if and only if $i\leq j$. Let $(a+id)_{i=0}^{k-1}$ be an arithmetic progression strongly contained in $A$. Given $i\in\Z$, set $b_i=(i-1)d$. For $1\leq i\leq k$, define $c_i=a+b_i$. Given $1\leq i,j\leq k$, we have $c_j-b_i\in A$ if and only if $i\leq j$. 
\end{proof}

\begin{definition}
A set $A\seq\N$ is \textbf{$\ap_*$-sparse} if, for all $n\geq 1$, $\Sigma_n(A)$ does not strongly contain arbitrarily long arithmetic progressions.
\end{definition}

From Proposition \ref{prop:strcon}, we immediately obtain the following conclusion.

\begin{theorem}\label{thm:stabap*}
If $A\seq\N$ is stable then $A$ is $\ap_*$-sparse.
\end{theorem}

It is not immediately obvious that $\ap_*$-sparse sets are actually sparse in any reasonable way. For example, $\Z$ does not strongly contain arbitrarily long arithmetic progressions. The next remark gives some justification for our terminology and, once again, highlights the importance of restricting to subsets of $\Z$ which are either bounded above or bounded below.

\begin{remark}\label{rem:dimsparse}
Fix a set $A\seq\N$, and suppose $(a+id)_{i=0}^{n-1}$ is an arithmetic progression in $A$. If $a\leq d$ then $(a+id)_{i=0}^{n-1}$ is strongly contained in $A$, since $a-id<0$ for all $i\geq 1$. Any arithmetic progression of the form $(a+id)_{i=0}^{n-1}$ contains a sub-progression $(a+id_*)_{i=0}^{k-1}$ with $a\leq d_*$ and $k\geq n/\lceil\frac{a}{d}\rceil$. Therefore, if $A$ does not strongly contain arbitrarily long arithmetic progressions, then there is a uniform finite bound on the value of $n/\lceil\frac{a}{d}\rceil$ for any arithmetic progression $(a+id)_{i=0}^{n-1}$ appearing in $A$. It follows (via Remark \ref{rem:ENS}) that $\ap_*$-sparse subsets of $\N$ are $\udelta$-sparse.
\end{remark}

Figure \ref{figure} collects the properties of subsets of $\N$ that have been defined, along with the known relationships between them.

\begin{figure}[H]
\begin{displaymath}
\xymatrix{
 & \text{$\udelta$-sparse}& \\
 \text{$\ap_*$-sparse}\ar[ur]  &   \text{$\delta^b$-sparse}\ar[u]  & \\
 \text{$\ap$-sparse} \ar[u]\ar[ur] &  \text{stable}\ar[u]\ar[ul] &  \text{\begin{tabular}{c} sufficiently \\ sparse\end{tabular}}\ar[ul]\\
  & \text{\begin{tabular}{c} geometrically \\ sparse\end{tabular}}\ar[ur] \ar[u] &
}
\end{displaymath}
\caption{Known implications between properties of subsets of $\N$.}
\label{figure}
\end{figure}
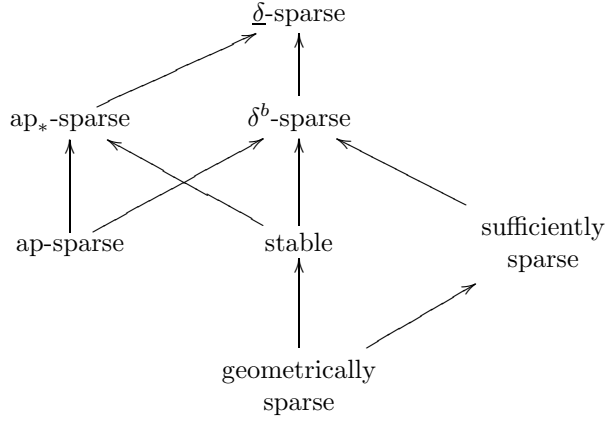

All of these implications are either trivial or proved explicitly above, except for ``sufficiently sparse implies $\delta^b$-sparse". Since this is not directly relevant to our results, we only sketch the argument and leave the details to the reader. Suppose $A\seq\N$ is not $\delta^b$-sparse. By \cite[Theorem 1]{JinBD}, there are $m,n\geq 1$ such that $mI\seq \Sigma_n(A)$ for arbitrarily large intervals $I\seq\N$.  Therefore $\Sigma_{2n}(\pm A)$ contains $m\Z$, and so $A$ is not sufficiently sparse. (Note that this also shows ``$\ap$-sparse implies $\delta^b$-sparse" without using Szemer\'{e}di's Theorem.)

We end this section with a few miscellaneous remarks and questions.

\subsection{Limits of ratios and additive bases}

Given a geometrically sparse set $A\seq\N$, a key ingredient in the proof of stability of $\Th(\cZ_A)$ was the associated sequence $(\lambda_m)_{m=0}^\infty$ such that $\inf_{m\geq0}\frac{\lambda_{m+1}}{\lambda_m}>1$.

\begin{definition}
Given an infinite set $A\seq\N$, define 
\[
\minf(A)=\liminf_{n\rightarrow\infty}\frac{a_{n+1}}{a_n}\mand\msup(A)=\limsup_{n\rightarrow\infty}\frac{a_{n+1}}{a_n},
\]
where $(a_n)_{n=0}^\infty$ is a monotonic enumeration of $A$. If $\minf(A)=\msup(A)$, then we let $\mfrak(A)=\lim_{n\rightarrow\infty}\frac{a_{n+1}}{a_n}$ denote this common value.
\end{definition}

Note that, given $A\seq\N$, $\minf(A)=\infty$ if and only if $\mfrak(A)=\infty$, and $\msup(A)=1$ if and only if $\mfrak(A)=1$. We have shown that if $\mfrak(A)=\infty$ then $A$ is stable (via Proposition \ref{prop:esgs}). Moreover, for any real number $r>1$, there is a stable set $A\seq\N$ such that $\mfrak(A)=r$ (take $a_n$ to be the integer part of $r^n$ as in Example \ref{ex:stratio}).

\begin{question}\label{ques:multbig}
Is there an unstable set $A\seq\N$ with $\minf(A)>1$? Is there a stable set $A\seq\N$ with $\mfrak(A)=1$?\footnote{In a later paper, we show that both of these questions have a positive answer. See \cite{CoGG}.}
\end{question}

 It is of course easy to find unstable sets $A\seq\N$ such that $\mfrak(A)=1$ (e.g. $A=\N$).  The inquiry into such sets leads to classical results around \emph{additive bases}, which are sets $A\seq\N$ such that $\Sigma_n(A)=\N$ for some $n\geq 1$ (see \cite{NathB1}). We broaden this definition as follows. 

\begin{definition}
A set $A\seq\N$ is a \textbf{relative additive asymptotic base} (RAAB) if there is some $n\geq 1$ such that $\Sigma_n(A)$ is cofinite in $a\N$, where $a=\gcd(A)$. 
\end{definition}

The following observation follows from the proof of Theorem \ref{thm:stabudelta}.

\begin{corollary}\label{cor:RAAB}
If $A\seq\N$ is a RAAB then $\cZ_A$ defines $\N$. 
\end{corollary}

In fact, by the result of Nash and Nathanson \cite{NaNa} mentioned in Remark \ref{rem:NaNa}, $A\seq\N$ is a RAAB if and only if $A$ is not $\udelta$-sparse. Therefore Corollary \ref{cor:RAAB} is precisely the contrapositive of Theorem \ref{thm:stabudelta}. By definition, if $A$ is a RAAB then $\mfrak(\Sigma_n(A))=1$ for some $n>0$. In 1770, Waring conjectured that for any $k\geq 1$, the set $P_k=\{n^k:n\in\N\}$ is an additive base. This is trivial for $k=1$, and a famous result of Lagrange from 1770 is that $\Sigma_4(P_2)=\N$. In 1909, Hilbert \cite{HilbertWaring} proved Waring's conjecture for all $k$. The fact that $P_k$ is a RAAB for all $k\geq 1$ is generalized by the following result.

\begin{fact}[Kamke \cite{Kamke} 1931]
Let $f(x)$ be a non-constant polynomial such that $A_f:=\{f(n):n\in\N\}\seq\N$. Then $A_f$ is a RAAB.
\end{fact}

An example of a RAAB, which is not covered by Kamke's result, is the set $P$ of primes (the existence of such an $n>0$ such that $\Sigma_n(P)=\N$ was first established by Schnirel'mann \cite{Schn} in 1933).

\subsection{Other notions of model theoretic tameness}

We currently have no concrete example of an unstable set $A\seq\N$ such that $\cZ_A$ does not define the ordering on $\Z$ (although Theorems \ref{thm:stabBD} and \ref{thm:stabap*} generate possible candidates). So it is natural to ask if stability of $\Th(\cZ_A)$ can fail in more subtle ways. For instance:

\begin{question}
 Is there an unstable subset $A\seq\N$ such that $\Th(\cZ_A)$ is simple?
\end{question}

Again, the restriction to subsets of $\N$ is important here. For example, using results of Chatzidakis and Pillay \cite{ChzPi} on ``generic" predicates, one can prove the existence of $A\seq\Z$ such that $\Th(\cZ_A)$ is supersimple (of $SU$-rank $1$) and unstable. However, such sets given by \cite{ChzPi} are necessarily unbounded above and below. The following is another result in this vein by Kaplan and Shelah.

\begin{fact}\textnormal{\cite{KaSh}}\label{fact:KaSh}
If $P$ is the set of primes then $\Th(\cZ_{\pm P})$ is unstable and, assuming Dickson's Conjecture\footnote{Dickson's Conjecture, which concerns primes in arithmetic progressions, is a fairly strong conjecture in number theory, which implies several other conjectures and known results (e.g. infinitely many twin primes, infinitely many Sophie Germain primes, and the Green-Tao Theorem).}, $\Th(\cZ_{\pm P})$ is supersimple of $SU$-rank $1$. 
\end{fact}

 In general, if $A\seq\Z$ is stable then $\Th(\cZ_{\pm A})$ is stable, as $\pm A$ is definable in $\cZ_A$. There are also clearly unstable subsets $A$ such that $\Th(\cZ_{\pm A})$ is stable (e.g. $A=\N$). The following is an example of an unstable set $A\seq\N$ such that $\cZ_A$ and $\cZ_{\pm A}$ are interdefinable.

\begin{example}
Let $A$ be the set of squares (recall that $A$ is an asymptotic base and therefore unstable). We show that $A$ is definable in $\cZ_{\pm A}$. For any odd $n\in\N$, we have $n^2\equiv_4 1$ and $\nv n^2\equiv_4 3$. Thus $B:=(\pm A)\cap(4\Z+1)$, which is definable in $\cZ_{\pm A}$, is precisely the set of positive odd squares (so $B\seq\N$). By classical work of Gauss, $8\N+3\seq\Sigma_3(B)$ (see, e.g., \cite[Theorem 1.5]{NathB1}). By Theorem \ref{thm:stabudelta}, $\cZ_{\pm A}$ defines the ordering, and so $A$ is definable in $\pm A$.
\end{example}

\begin{problem}
Characterize the unstable subsets $A\seq\N$ such that $\Th(\cZ_{\pm A})$ is stable (or simple).
\end{problem}

\section*{Acknowledgements} 

I would like to thank Chris Laskowski, Daniel Palac\'{i}n, Rizos Sklinos, Kyle Gannon, and Leo Jimenez for several beneficial conversations. I also thank Bjorn Poonen for allowing me to use \cite{PooMO}, and Isaac Goldbring for allowing me to include Theorem \ref{thm:stabBD}.

A previous draft of this paper included the conjecture that if $A\seq\N$ is stable and $B\seq A$, then $B$ is stable. This has been removed due to a straightforward counterexample found by Jimenez. 

Independently of our work, Lambotte and Point \cite{LaPo} have also studied stable expansions of $(\Z,+,0)$ by unary predicates, and some of our examples of geometrically sparse sets described in Section \ref{sec:GS} overlap with examples considered in \cite{LaPo}.

\end{document}